\pdfoutput=1
\documentclass{scrartcl}

\usepackage[utf8]{inputenc}
\usepackage{tkz-euclide}
\usetikzlibrary{calc}
\usepackage{pgfplots}
\usepackage{multicol}
\usepackage{comment}
\usepackage{amsthm}
\usepackage{xcolor}
\usepackage{enumitem}
\usepackage{hyperref}

\usetikzlibrary{patterns}

\newcommand*\diff{\mathop{}\!\mathrm{d}}

\usepackage{amsmath,amssymb,dsfont,mathtools,booktabs,amsthm}

\newtheorem{theorem}{Theorem}
\newtheorem{remark}{Remark}
\newtheorem{corollary}{Corollary}

\pgfplotsset{compat=1.18}
\title{A priori error analysis of multirate time-stepping schemes for two-phase flow problems}

\author{Martyna Soszy\~nska \and Thomas Richter\thanks{Otto-von-Guericke University Magdeburg, \url{thomas.richter@ovgu.de}}}

\begin{document}

\maketitle

\begin{abstract}
  We present a priori error estimates for a multirate time-stepping scheme for coupled differential equations. The discretization is based on Galerkin methods in time using two different time meshes for two parts of the problem. We aim at surface coupled multiphysics problems like two-phase flows. Special focus is on the handling of the interface coupling to guarantee a coercive formulation as key to optimal order error estimates.

  In a sequence of increasing complexity, we begin with the coupling of two ordinary differential equations, coupled heat conduction equation, and finally a coupled Stokes problem. For this we show optimal multi-rate estimates in velocity and a suboptimal result in pressure. 

  The a priori estimates prove that the multirate method decouples the two subproblems exactly. This is the basis for adaptive methods which can choose optimal lattices for the respective subproblems.
\end{abstract}

\section{Introduction}

 Many multiphysics problems such as two-phase flows~\cite{Gross2006}, fluid-structure interactions~\cite{Richter2017,Frei2020}, or
chemically reactive flows~\cite{Zhou2017} have a multiscale character in
time. One prominent  example is found in fluid-structure interactions. Here the time scales of fluid and solid often strongly differ or time-steps should be taken differently for reasons of numerical
stability~\cite{Moerloose2018}. Such multirate time-stepping schemes are often used in applications and the theory is well developed in particular for coupled ODE problems, see~\cite{Gander2013} for an overview.

 Considering partial differential equations multiple temporal scales
can either appear within the same domain or they appear on disjoint
domains that are coupled along a common interface boundary. The first
type of problem is called volume-coupled multiphysics problems and
examples with a temporal multiscale character are chemically reactive
flows or transport processes in subsurface
flows. Prototypical examples for the second kind of problem,
surface-coupled multiphysics problems, are fluid-structure
interactions or multiphase flow problems. The latter kind, although in
a simple linear setting, is discussed in this manuscript.

 In our previous work~\cite{Soszynska2021} we have presented a
variational multirate framework for the space-time discretization of
coupled systems of partial differential equations. Here we introduced
an efficient iteration for the fast solution of coupled time-steps that
is based on the idea of the shooting method. Further, we have
introduced an a posteriori error estimator that also allows us to
adaptively pick optimal time steps for slow and fast problems. In this work, we focus on a priori stability and error estimates for multirate discretizations of coupled temporal multiscale problems. 

\paragraph*{Outline} In the following sections, we develop an a priori error estimate for a multirate time-stepping scheme based on a
discontinuous Galerkin representation of the backward Euler
method.  Since the arguments leading to the error estimate build on each other, we start with very simple ODE problems. Here we introduce the basic notation and explain our handling of the coupling of different time lattices. Section 3 is then devoted to a coupled system of two heat conduction equations with a jump in diffusivity. At its core is the space-time variational formulation of the PDE system and, in particular, how to deal with the boundary conditions at the coupling boundary when deriving the error estimates. This section already contains the central arguments. Finally, a system of two Stokes equations with different viscosities is considered in Section 4. A summary and outlook is given in Section 5.


\section{Coupling of Ordinary Differential Equations}

 As our first example we chose a system of two
ordinary equations where we look for a solution $\textbf{u} = (u_1,
u_2)^T: I \times I\to \mathds{R} \times \mathds{R} $ with $\textbf{u} \in C^1(\bar I)^2$ to a system
\begin{equation}
  d_t u_1(t) = f_1\big(t,u_1(t),u_2(t)\big),\quad
  d_t u_2(t) = f_2\big(t,u_1(t),u_2(t)\big),\quad \textbf{u}(0)=\textbf{0}.
\label{continuous_ode}
\end{equation}
We assume that the function $\textbf{f}=(f_1, f_2)$ is Lipschitz
continuous. Next,
we proceed to a semi-discrete formulation. The two problems are
coupled across the macro mesh over the time interval $I=[0, T]$

\begin{equation*}
    0 = t^0 < t^1 ... <t^N = T, \hspace*{0.2 cm} k^n = t^n - t^{n - 1}, \hspace*{0.2 cm} I^n = (t^{n -1}, t^n].
\end{equation*}
Every time subinterval $I^n$ has its own time interval partitioning corresponding to each of the subproblems
\begin{equation}
  t^{n-1}=t^{n,0}_j<t^{n,1}_j<\dots<t^{n,N_j^n}_j=t^n,\; k^{n}_j = t^{n,m}_j-t^{n,m-1}_j\quad j=1,2.
\end{equation}
as well as 
\begin{equation*}
    k_1 \coloneqq \max_{n} k_1^{n}, \hspace*{0.5 cm} k_2 \coloneqq \max_{n} k_2^{n}, \hspace*{0.5 cm} k \coloneqq \max \left\{ k_1, k_2\right\}.
\end{equation*}
The micro triangulations are uniform on each $I^n$, which is, however, just for simplicity of notation. We denote the time mesh by ${\cal I}_k$. 
An example of this kind of time mesh is shown in Figure \ref{time-partitioning}. On ${\cal I}_k$ we define the space of piecewise continuous functions
\[
C_j({\cal I}_k) = \{ \phi\in L^2(I)\big|\, u|_{I_j^{n,m}}\in C^1(I_j^{n,m})\}\; j=1,2,\quad C({\cal I}_k) = C_1({\cal I}_k)\times C_2({\cal I}_k).
\]

 We assume that in this time partitioning we introduce micro time-steps only when necessary, see Figure \ref{time-partitioning-refined}. Besides, we assume that these time meshes are a result of an adaptive time-stepping procedure where time-steps are refined only in the middle. Based on these two assumptions for every macro time-step $I^n$ we either have $N_1^n = 1$ or $N_2^n = 1$. We choose discrete solutions $\textbf{u}^k\in X^k=X^k_1\times X^k_2$ as piecewise constant functions defined over each of the meshes
\begin{equation}\label{spaces}
    X_{j}^k ={} \big\{ u \in L^2(\bar{I})\Big|\; u|_{I^{n,m}_j} \in \mathbb{R} 
      \text{ for all }
    I^{n,m}_j \subset I  \text{ and } u(0)=0 \big\},\quad j=1,2.
\end{equation}
To further specify the time-stepping scheme, we define the operator $\textbf{i}^k = (i_1^k, i_2^k)^T$  with $\textbf{i}^k: C({\cal I}_k)\to X^k$.
by
\begin{equation}
    i_j^k u \big|_{I^{n, m}_j} \coloneqq u(t^{n,m}_j),\quad j=1,2.
\label{ie_operator}
\end{equation}
At initial time, we set $(\textbf{i}^k\textbf{u})(0) = \textbf{u}(0)$. This choice of projection operators indicates the implicit Euler method. We follow this route and introduce a finite difference quotient $d_t^k$ typical for this time-stepping scheme
\begin{equation}
  d_t^k u^k_j \Big|_{I_j^{n,m}} = \frac{u^k_j(t_j^{n,m}) - u_j^k(t_j^{n, m-1})}{k_j^{n,m}},\quad j=1,2.\label{time_derivative}
\end{equation}

  \begin{figure}[t]
    \centering
    \begin{tikzpicture}[scale = 1.2]
        
        \draw[very thick] (0.0, 1.0) -- (8.0, 1);
        \draw[very thick] (0.0, 1.75) -- (0.0, 0.25);
        \draw[very thick] (4.0, 1.75) -- (4.0, 0.25);
        \draw[very thick] (8.0, 1.75) -- (8.0, 0.25);
        \draw[very thick] (6.0, 1.0) -- (6.0, 0.5);
        \draw[very thick] (1.0, 1.0) -- (1.0, 1.5);
        \draw[very thick] (2.0, 1.0) -- (2.0, 1.5);
        \draw[very thick] (3.0, 1.0) -- (3.0, 1.5);
        \draw[<->, dashed] (4.1, 0.75) -- (5.9, 0.75); 
        \node at (5.1, 0.4) {$k^{n, 1}_2$};
        \node at (7.1, 0.4) {$k^{n, 2}_2$};
        \draw[<->, dashed] (6.1, 0.75) -- (7.9, 0.75); 
        \draw[<->, dashed] (0.1, 0.0) -- (3.9, 0.0); 
        \draw[<->, dashed] (4.1, 0.0) -- (7.9, 0.0); 
        \node at (2.1, -0.4) {$k^{n - 1}$};
        \node at (6.0, -0.4) {$k^{n}$};
        \node at (0.0, 2.0) {$t^{n - 2}$};
        \node at (4.0, 2.0) {$t^{n - 1}$};
        \node at (8.0, 2.0) {$t^{n}$};
        
        \node at (1.1, 1.75) {$t^{n - 1, 1}_1$};
        \node at (2.1, 1.75) {$t^{n - 1,2}_1$};
        \node at (3.1, 1.75) {$t^{n-1, 3}_1$};
    
    \end{tikzpicture}
    \caption{We show a snapshot of time partitioning with two macro time-steps. In the first macro time-step, we introduce four micro time-steps in the first subproblem. In the second one, we have two micro time-steps in the second subproblem.}
    \label{time-partitioning}
  \end{figure}
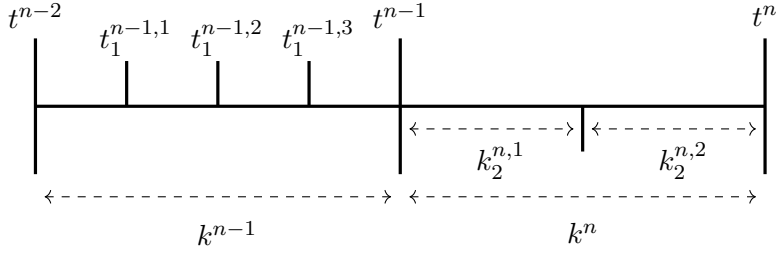

  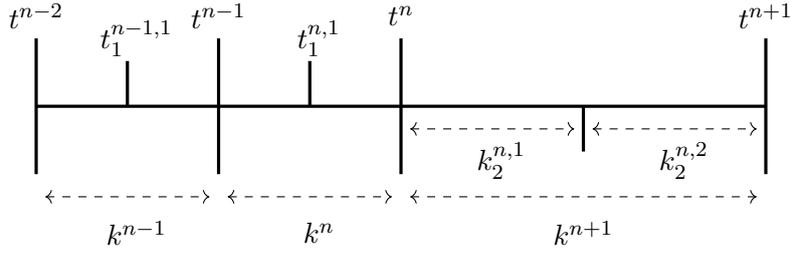
\begin{figure}[t]
    \centering
    \begin{tikzpicture}[scale = 1.2]
        
        \draw[very thick] (0.0, 1.0) -- (8.0, 1);
        \draw[very thick] (0.0, 1.75) -- (0.0, 0.25);
        \draw[very thick] (4.0, 1.75) -- (4.0, 0.25);
        \draw[very thick] (8.0, 1.75) -- (8.0, 0.25);
        \draw[very thick] (6.0, 1.0) -- (6.0, 0.5);
        \draw[very thick] (1.0, 1.0) -- (1.0, 1.5);
        \draw[very thick] (2.0, 1.75) -- (2.0, 0.25);
        \draw[very thick] (3.0, 1.0) -- (3.0, 1.5);
        \draw[<->, dashed] (4.1, 0.75) -- (5.9, 0.75); 
        \node at (5.1, 0.4) {$k^{n, 1}_2$};
        \node at (7.1, 0.4) {$k^{n, 2}_2$};
        \draw[<->, dashed] (6.1, 0.75) -- (7.9, 0.75); 
        \draw[<->, dashed] (0.1, 0.0) -- (1.9, 0.0); 
        \draw[<->, dashed] (2.1, 0.0) -- (3.9, 0.0); 
        \draw[<->, dashed] (4.1, 0.0) -- (7.9, 0.0); 
        \node at (1.1, -0.4) {$k^{n - 1}$};
        \node at (3.1, -0.4) {$k^{n}$};
        \node at (6.0, -0.4) {$k^{n +1}$};
        \node at (0.0, 2.0) {$t^{n - 2}$};
        \node at (2.0, 2.0) {$t^{n - 1}$};
        \node at (4.0, 2.0) {$t^{n}$};
        \node at (8.0, 2.0) {$t^{n+1}$};
        
        \node at (1.1, 1.75) {$t^{n - 1, 1}_1$};
        \node at (3.1, 1.75) {$t^{n, 1}_1$};
    
    \end{tikzpicture}
    \caption{Here we show the refinement of the time mesh presented in Figure \ref{time-partitioning} by splitting $I^n$ in the second subproblem. Since the time mesh corresponding to the first subproblem was already split at this point (point $t^{n - 1,2}_1$ in Figure~\ref{time-partitioning}), we can introduce a new macro time-step.}
    \label{time-partitioning-refined}
  \end{figure}
 Since we are interested in coupled problems, we also need an apparatus to deal with the transfer of solutions between the non-matching time meshes. To resolve this issue, we additionally introduce an operator $\textbf{I}^k = (I_1^k, I_2^k)^T$  with $\textbf{I}^k: C({\cal I}) \to X^k$
given by an average over each partitioning, where on every interval $I^{n, m}_1$ and $I^{n, m}_2$ it holds
\begin{equation}
    I_j^k u \big|_{I^{n, m}_j} \coloneqq \frac{1}{k_j^{n, m}} \int_{I^{n, m}_j} u \diff t, \quad j=1,2.
\label{average_operator}
\end{equation}
We will also use a similar operator $\bar{I}^k$ which is defined accordingly on the coarse time mesh consisting of the macro time-steps
\begin{equation}
    \bar{I}^k u \big|_{I^{n}} \coloneqq \frac{1}{k^{n}} \int_{I^{n}} u \diff t.
\label{average_operator_global}
\end{equation}
At the initial point, we impose $(\textbf{I}^k\textbf{u})(0) = \textbf{u}(0)$ and $(\bar{I}^k \textbf{u})(0) = \textbf{u}(0)$. Further information is given in Figure \ref{interface_operator_transfer}. The key property of the operator $\textbf{I}^k$ is that its error has average zero over each macro time-step
\begin{equation*}
    \int_{I^n} \left(\textbf{u} - \textbf{I}^k \textbf{u} \right) \diff t = 0.
\end{equation*}
Moreover, for any $u_1^k \in X^k_1$ and $u_2^k \in X^k_2$ we similarly have 
\begin{equation}
    \int_{I^n} \left(u_1^k - I_2^k u_1^k \right) \diff t =\int_{I^n} \left(u_2^k - I_1^k u_2^k \right) \diff t = 0.
\label{average_zero}
\end{equation}
This identity directly follows from the hierarchical structure in the discretization, where we know that nodes $t^{n}$ and $t^{n-1}$ belong to both of the discretizations. 

\begin{figure}[t]
\centering
\begin{tikzpicture}[scale = 0.9]
\draw[->, thick] (-0.5, 0) -- (4, 0);
\draw[->, thick] (0, -0.5) -- (0, 4);
\draw[pattern=north west lines, dashed] (0.5,0.0) rectangle (3.5, 2);
\draw (-0.15, 2.0) -- (0.15, 2.0);
\node at (-0.8, 2.0) {$u_1^k(t^n)$};
\node at (0.7, -0.5) {$t^{n - 1}$};
\node at (3.5, -0.5) {$t^{n}$};
\draw[line width = 0.04 cm] (0.5, 2) -- (3.5, 2);
\draw[fill=black] (3.5,2) circle (2 pt);

\node at (4.8, 1.5) {\Large{$\longmapsto$}};
\node at (4.8, 2.0) {\large{$I_2^k$}};

\draw (7.35, 2.0) -- (7.65, 2.0);
\node at (6.4, 2.0) {$(I^k_2 u_1^k)(t^n)$};
\draw[->, thick] (7.0, 0) -- (11.5, 0);
\draw[->, thick] (7.5, -0.5) -- (7.5, 4);
\draw[pattern=north west lines, dashed] (8.0,0) rectangle (9.5, 2);
\draw[pattern=north west lines, dashed] (9.5, 2) -- (11.0, 2.0) -- (11.0, 0.0) -- (9.5, 0.0);
\draw[fill=black] (11.0,2) circle (2 pt);
\draw[fill=black] (9.5,2) circle (2 pt);
\draw[line width = 0.04 cm] (8.0,2) -- (11.0,2);
\node at (8.1, -0.5) {$t^{n - 1}$};
\node at (11.0, -0.5) {$t^{n}$};
\end{tikzpicture}

\begin{tikzpicture}[scale = 0.9]
\draw[->, thick] (-0.5, 0) -- (4, 0);
\draw[->, thick] (0, -0.5) -- (0, 4);
\draw[pattern=north west lines, dashed] (0.5, 0.0) -- (2.0, 0.0) -- (2.0, 3.0) -- (0.5, 3.0) -- (0.5, 0.0);
\draw (-0.15, 3.0) -- (0.15, 3.0);
\node at (-0.8, 1.0) {$u_1^k(t^n)$};
\draw (-0.15, 1.0) -- (0.15, 1.0);
\node at (0.7, -0.5) {$t^{n - 1}$};
\node at (3.5, -0.5) {$t^{n}$};
\draw[pattern=north west lines, dashed] (2.0, 1.0) -- (3.5, 1.0) -- (3.5, 0.0) -- (2.0, 0.0);
\draw[line width = 0.04 cm] (2.0, 1) -- (3.5, 1);
\draw[fill=black] (3.5,1) circle (2 pt);
\draw[fill=black] (2,3) circle (2 pt);
\draw[line width = 0.04 cm] (2, 3) -- (0.5, 3);

\node at (4.8, 1.5) {\Large{$\longmapsto$}};
\node at (4.8, 2.0) {\large{$I_2^k$}};

\draw (7.35, 2.0) -- (7.65, 2.0);
\node at (6.4, 2.0) {$(I^k_2 u_1^k)(t^n)$};
\draw[->, thick] (7.0, 0) -- (11.5, 0);
\draw[->, thick] (7.5, -0.5) -- (7.5, 4);
\draw[pattern=north west lines, dashed] (8.0,0) rectangle (11.0, 2);
\draw[fill=black] (11.0,2) circle (2 pt);
\draw[line width = 0.04 cm] (8.0,2) -- (11.0,2);
\node at (8.1, -0.5) {$t^{n - 1}$};
\node at (11.0, -0.5) {$t^{n}$};
\end{tikzpicture}
\caption{We present two examples of the transformation given by the interface projection operator $\textbf{I}^k$. In the top sketch, one macro time-step is split into two smaller micro time-steps and $(I^k_2 u_1^k)(t^{n, 1}_2) = (I^k_2 u_1^k)(t^{n}) = u_1^k(t^{n})$. In the bottom sketch, two smaller micro time-steps are merged together with $(I^k_2 u_1^k)(t^{n}) = \frac{k^{n, 1}_1}{k^{n}} u_1^k(t^{n, 1}_1) + \frac{k^{n, 2}_1}{k^{n}} u_1^k(t^{n})$.}
\label{interface_operator_transfer}
\end{figure}
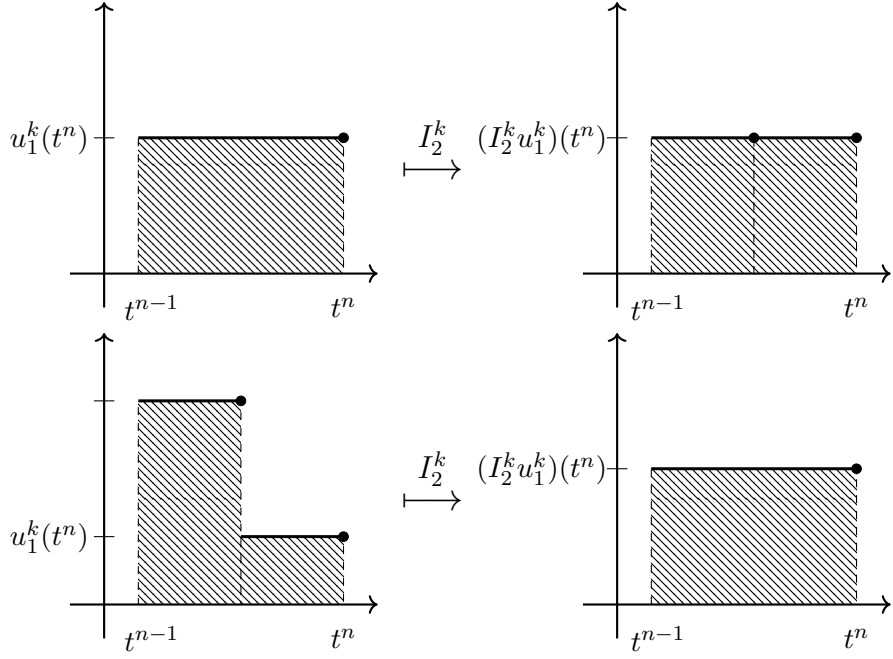

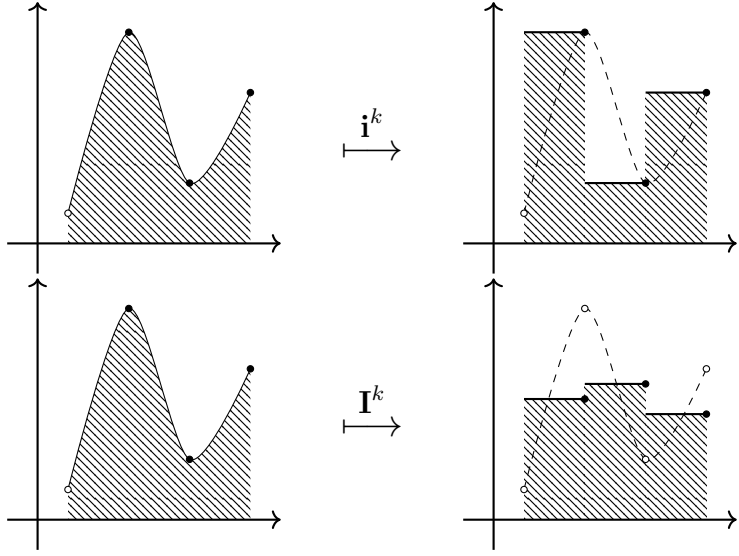
\begin{figure}
    \centering
\begin{tikzpicture}[scale = 0.8]
    \draw[->, thick] (-0.5, 0) -- (4, 0);
    \draw[->, thick] (0, -0.5) -- (0, 4);
    
    \draw plot [smooth] coordinates {(0.5,0.5)  (1.5,3.5)  (2.5,1)  (3.5,2.5)}; 
    \fill[pattern=north west lines, dashed] plot [smooth] coordinates {(0.5,0.5)  (1.5,3.5)  (2.5,1)  (3.5,2.5)} -- (3.5,0.0) -- (0.5,0.0); 

    \node (a0) at (0.5,0.5) {};
    \draw[fill = white] (a0) circle [radius=1.5pt];
    \node (a1) at (1.5,3.5) {};
    \draw[fill] (a1) circle [radius=1.5pt];
    \node (a2) at (2.5,1) {};
    \draw[fill] (a2) circle [radius=1.5pt];
    \node (a3) at (3.5,2.5) {};
    \draw[fill] (a3) circle [radius=1.5pt];       
    
    \node at (5.5, 1.5) {\Large{$\longmapsto$}};
    \node at (5.5, 2.0) {\large{$\textbf{i}^k$}};
    
    \draw[->, thick] (7.0, 0) -- (11.5, 0);
    \draw[->, thick] (7.5, -0.5) -- (7.5, 4);
    
    \draw[dashed] plot [smooth] coordinates {(8.0,0.5)  (9.0,3.5)  (10.0,1)  (11.0,2.5)};
    
    \draw[thick] (8.0, 3.5) -- (9.0, 3.5);
    \fill[pattern=north west lines, dashed] (8.0, 3.5) -- (9.0, 3.5) -- (9.0, 0.0) -- (8.0, 0.0) -- (8.0, 3.5);
    \draw[thick] (9.0, 1.0) -- (10.0, 1.0);
    \fill[pattern=north west lines, dashed] (9.0, 1.0) -- (10.0, 1.0) -- (10.0, 0.0) -- (9.0, 0.0) -- (9.0, 1.0);
    \draw[thick] (10.0, 2.5) -- (11.0, 2.5);
    \fill[pattern=north west lines, dashed] (10.0, 2.5) -- (11.0, 2.5) -- (11.0, 0.0) -- (10.0, 0.0) -- (10.0, 2.5);

    \node (b0) at (8.0,0.5) {};
    \draw[fill = white] (b0) circle [radius=1.5pt];
    \node (b1) at (9.0,3.5) {};
    \draw[fill] (b1) circle [radius=1.5pt];
    \node (b2) at (10.0,1) {};
    \draw[fill] (b2) circle [radius=1.5pt];
    \node (b3) at (11.0,2.5) {};
    \draw[fill] (b3) circle [radius=1.5pt];

\end{tikzpicture}

\begin{tikzpicture}[scale = 0.8]
    \draw[->, thick] (-0.5, 0) -- (4, 0);
    \draw[->, thick] (0, -0.5) -- (0, 4);
    
    \draw plot [smooth] coordinates {(0.5,0.5)  (1.5,3.5)  (2.5,1)  (3.5,2.5)}; 
    \fill[pattern=north west lines, dashed] plot [smooth] coordinates {(0.5,0.5)  (1.5,3.5)  (2.5,1)  (3.5,2.5)} -- (3.5,0.0) -- (0.5,0.0); 

    \node (a0) at (0.5,0.5) {};
    \draw[fill = white] (a0) circle [radius=1.5pt];
    \node (a1) at (1.5,3.5) {};
    \draw[fill] (a1) circle [radius=1.5pt];
    \node (a2) at (2.5,1) {};
    \draw[fill] (a2) circle [radius=1.5pt];
    \node (a3) at (3.5,2.5) {};
    \draw[fill] (a3) circle [radius=1.5pt];       
    
    \node at (5.5, 1.5) {\Large{$\longmapsto$}};
    \node at (5.5, 2.0) {\large{$\textbf{I}^k$}};
    
    \draw[->, thick] (7.0, 0) -- (11.5, 0);
    \draw[->, thick] (7.5, -0.5) -- (7.5, 4);
    
    \draw[dashed] plot [smooth] coordinates {(8.0,0.5)  (9.0,3.5)  (10.0,1)  (11.0,2.5)};
    
    \draw[thick] (8.0, 2.0) -- (9.0, 2.0);
    \fill[pattern=north west lines, dashed] (8.0, 2.0) -- (9.0, 2.0) -- (9.0, 0.0) -- (8.0, 0.0) -- (8.0, 2.0);
    \draw[thick] (9.0, 2.25) -- (10.0, 2.25);
    \fill[pattern=north west lines, dashed] (9.0, 2.25) -- (10.0, 2.25) -- (10.0, 0.0) -- (9.0, 0.0) -- (9.0, 2.25);
    \draw[thick] (10.0, 1.75) -- (11.0, 1.75);
    \fill[pattern=north west lines, dashed] (10.0, 1.75) -- (11.0, 1.75) -- (11.0, 0.0) -- (10.0, 0.0) -- (10.0, 1.75);

    \node (b0) at (8.0,0.5) {};
    \draw[fill = white] (b0) circle [radius=1.5pt];
    \node (b1) at (9.0,3.5) {};
    \draw[fill = white] (b1) circle [radius=1.5pt];
    \node (b2) at (10.0,1) {};
    \draw[fill = white] (b2) circle [radius=1.5pt];
    \node (b3) at (11.0,2.5) {};
    \draw[fill = white] (b3) circle [radius=1.5pt];  
    
    \node (c1) at (9.0,2.0) {};
    \draw[fill] (c1) circle [radius=1.5pt];
    \node (c2) at (10.0,2.25) {};
    \draw[fill] (c2) circle [radius=1.5pt];
    \node (c3) at (11.0, 1.75) {};
    \draw[fill] (c3) circle [radius=1.5pt];

\end{tikzpicture}

    \caption{An example showing a difference between projection operators $\textbf{i}^k$ and $\textbf{I}^k$. In the top figure, we can see the projection given by the $\textbf{i}^k$ operator and in the bottom one, we instead have a look at the $\textbf{I}^k$ operator.}
    \label{projection}
\end{figure}


 We would like to reiterate the difference between operators $\textbf{i}^k$ and $\textbf{I}^k$. The former is our primary operator used in the implicit Euler time-stepping scheme and will be a part of a projection error that we will estimate in each of the following proofs. The latter is used exclusively to transfer the solutions between different time meshes. The difference between the two is further illustrated in Figure \ref{projection}. Finally, both $\textbf{i}^k$ and $\textbf{I}^k$ are properly defined projection operators and therefore, for any $\textbf{u}^k \in X^k$, we have

\begin{equation*}
  \textbf{i}^k \textbf{u}^k = \textbf{I}^k \textbf{u}^k = \textbf{u}^k.
\end{equation*}


Given this preliminary information, we can define the semi-discrete problem
\begin{equation}
  u^k\in C({\cal I}):\quad 
  d_t^k u_1^k = f_1 (i_1^k t, u_1^k, I_1^k u_2^k), \quad
  d_t^k u_2^k  = f_2 (i_2^k t, I_2^k u_1^k, u_2^k). 
  \label{semi-discrete_ode}
\end{equation}
We will now prove a stability estimate of this semi-discrete system using Gronwall's lemma.

\begin{theorem}
  Let $\textbf{u}$ be a continuous solution to~(\ref{continuous_ode}) and $\textbf{u}^k \in X^k$ its discrete counterpart and a solution to~(\ref{semi-discrete_ode}).
  Further, let us assume that $\textbf{f} \in C^1(\bar{I})^2$ with Lipschitz constants $L_1$ and $L_2$, respectively.  If we further assume that $(L_1 + L_2)( k^N_1 + k^N_2)\leq \frac{1}{2}$, where $k^N_1 \coloneqq k^{N, N^N_1}_1$  and $k^N_2 \coloneqq k^{N, N^N_2}_2$ that is the sizes of the last time-steps in each of the timelines, then the following estimate holds
  \begin{equation*}
    \begin{aligned}
      & \big| \big| e_1^k(t^N)\big|\big| + \big| \big| e_2^k(t^N)\big|\big| \leq e^{2T(L_1 + L_2)} \left(2\|\tau_1^k\| + 2\|\tau_2^k\| \right),
    \end{aligned}
  \end{equation*}
  with the truncation errors (for $j=1,2$ and using the notation $\hat j=3-j$)
  \begin{multline}\label{tau:ode}
    \|\tau_j^k\|  \leq  \sum_{n = 1}^{N}\sum_{m = 1}^{N^n_j}  \Bigg\{  \frac{1}{2}(k^{n, m}_j)^2\max_{t \in I}\left| \left| d_t f_j(t, u_1, u_2) \right| \right| + L_j k^n k^{n, m}_j \max_{t \in I} \left| \left| f_j(t, u_1, u_2) \right| \right| \\
    + L_{\hat j} (k^{n, m}_j)^2  \max_{t \in I}\left|\left|f_{ j}(t, u_1, u_2)\right|\right|\Bigg\} 
  \end{multline}
  and where the errors $\textbf{e}^k = (e_1^k, e_2^k)^T$ are defined as
  \begin{flalign*}
    e_j^k \coloneqq u_j^k - i_j^k u_j, \quad j=1,2.
  \end{flalign*}
  \label{theorem_ode}
\end{theorem}
\begin{proof} Since the analysis of both the errors $e_1^k$ and $e_2^k$ is analogous, in this proof we will only estimate $e_1^k$. Using both the continuous (\ref{continuous_ode}) as well as the discrete (\ref{semi-discrete_ode}) formulation we have

\begin{flalign*}
e_1^k(t^{n,m}_1) = e_1^k(t^{n,m - 1}_1) + \int_{I^{n, m}_1} \Big\{ & f_1(t^{n,m}_1, u_1^k, I_1^k u_2^k)  - f_1(t^{n,m}_1,  u_1(t^{n,m}_1), \bar{I}^k u_2) \Big\} \diff t - \tau^{n,m}_{1,k}, 
\end{flalign*}
where
\begin{flalign*}
\tau^{n,m}_{1,k} \coloneqq & u_1(t^{n,m}_1) - u_1(t^{n, m-1}_1)  - \int_{I^{n,m}_1} f_1(t^{n,m}_1, u_1(t^{n,m}_1), \bar{I}^k u_2)  \diff t.
\end{flalign*}
We sum up the values of the errors over the whole time interval 
\begin{multline}
  e_1^k(t^N) =   \sum_{n = 1}^{N} \sum_{m = 1}^{N^n_1} \int_{I^{n,m}_1} \Big\{  f_1(t^{n,m}_1, u_1^k, I_1^k u_2^k)  - f_1(t^{n,m}_1,  u_1(t^{n,m}_1), \bar{I}^k u_2) \Big\} \diff t\\
  - \sum_{n = 1}^{N}\sum_{m = 1}^{N^n_1} \int_{I^{n,m}_1} \tau^{n,m}_{1,k} \diff t.
 \label{error_equation}
\end{multline}
Here, we used $e_1^k(0) = 0$. We will use the notation
$$\tau^{k}_1 \coloneqq \sum_{n = 1}^{N}\sum_{m = 1}^{N^n_1} \int_{I^{n,m}_1} \tau^{n,m}_{1,k} \diff t .$$
We can apply the triangle inequality to equation (\ref{error_equation}) and use  Lipschitz continuity of the function $f_1$
\begin{equation}
    \begin{aligned}
      \|e_1^k(t^N)\| &\leq  \sum_{n = 1}^{N} \sum_{m = 1}^{N^n_1} k^{n,m}_1 L_1\|e_1^k(t^{n,m}_1)\| + \sum_{n = 1}^{N} \sum_{m = 1}^{N^n_1} \int_{I^{n,m}_1} L_1 \|I_1^k u_2^k - \bar{I}^k u_2 \| \diff t - \tau_1^k.
    \end{aligned}
    \label{gronwall_equation}
\end{equation}
We proceed with the estimation of the term $\sum_{m = 1}^{N^n_1} \int_{I^{n,m}_1} L_1 \|I_1^k u_2^k - \bar{I}^k u_2 \| \diff t$. Based on the definition of our time meshes, for each macro time-step $I^n$, there is micro time-stepping in only one of the submeshes. Therefore, it is sufficient to only consider the following possibilities: 
\begin{enumerate}
    \item There is no micro time-stepping in $I_1^k$, from which follows that $N^n_1 = 1$ and $I_1^k\Big|_{I^n} = \bar{I}^k\Big|_{I^n}$
      \begin{equation*}
        \begin{aligned}
          &\sum_{m = 1}^{N^n_1} \int_{I^{n,m}_1} L_1 \|I_1^k u_2^k - \bar{I}^k u_2 \| \diff t = \int_{I^n} L_1 \|\bar{I}^k u_2^k - \bar{I}^k u_2 \| \diff t \\
          &\qquad = L_1 \left| \left| \sum_{m = 1}^{N^n_2} k^{n,m}_2 u_2^k(t^{n,m}_2) - \int_{I^n} u_2(s) \diff s \right| \right|  \\
          & \qquad \leq L_1 \sum_{m = 1}^{N^n_2} k^{n,m}_2 \Big( \|e_2^k(t^{n,m}_2) \| +  \frac{1}{k^{n,m}_2}{\int_{I^{n,m}_2}}\left| \left|  u_2(t^{n,m}_2) -  u_2(s) \right| \right|\diff s \Big)\\
          & \qquad \leq \sum_{m = 1}^{N^n_2} L_1 k^{n,m}_2 \|e_2^k(t^{n,m}_2) \| + \sum_{m = 1}^{N^n_2} L_1 (k^{n,m}_2)^2 \max_{t \in I} \|f_2(t, u_1, u_2)\|.
        \end{aligned}
      \end{equation*}
    \item There is no micro time-stepping in $I_2^k$, from which follows that $N^n_2 = 1$ and $I_2^k\Big|_{I^n} = \bar{I}^k\Big|_{I^n}$
    \begin{equation*}
    \begin{aligned}
      & \sum_{m = 1}^{N^n_1} \int_{I^{n,m}_1} L_1 \|I_1^k u_2^k - \bar{I}^k u_2 \| \diff t = \sum_{m = 1}^{N^n_1} L_1 k^{n,m}_1 \left|\left| u_2^k(t^n) - \frac{1}{k^n} \int_{I^n} u_2(s) \diff s \right|\right| \\
      &\qquad =L_1 \left|\left| k^n u_2^k(t^n) - \int_{I^n} u_2(s) \diff s \right|\right| \\
      & \qquad \leq  L_1 k^{n} \|e_2^k(t^{n}) \| +  L_1 (k^{n})^2 \max_{t \in I} \|f_2(t, u_1, u_2)\| \\
      & \qquad = \sum_{m = 1}^{N^n_2} L_1 k^{n,m}_2 \|e_2^k(t^{n,m}_2) \| + \sum_{m = 1}^{N^n_2} L_1 (k^{n,m}_2)^2 \max_{t \in I} \|f_2(t, u_1, u_2)\|.
    \end{aligned}
    \end{equation*}
\end{enumerate}

We will continue with the estimation of the remaining term in the error equation~(\ref{error_equation})



\begin{equation*}
\begin{aligned}
  \|\tau_1^k\| &\leq  
  \sum_{n=1}^{N}\sum_{m=1}^{N^n_1} \Bigg|\Bigg|u_1(t^{n,m}_1) - u_1(t^{n,m-1}_1) - k^{n,m}_1f_1(t^{n,m}_1, u_1(t^{n,m}_1), u_2(t^{n,m}_1)) \Bigg|\Bigg| \\
  & +  \sum_{n=1}^{N}\sum_{m=1}^{N^n_1} \int_{I^{n,m}_1}\Bigg|\Bigg|f_1(t^{n,m}_1, u_1(t^{n,m}_1), u_2(t^{n,m}_1))  - f_1( t^{n,m}_1, u_1(t^{n,m}_1), \bar{I}^k u_2)\Bigg|\Bigg| \diff t. 
\end{aligned}
\end{equation*}
We further have 
\begin{equation*}
 \Bigg|\Bigg|u_1(t^{n,m}_1) - u_1(t^{n,m-1}_1) - k^{n,m}_1f_1(t^{n,m}_1, u_1(t^{n,m}_1), u_2(t^{n,m}_1)) \Bigg|\Bigg|
 \leq  \frac{1}{2} (k^{n,m}_1)^2 \max_{t \in I} \| d_t f_1(t, u_1, u_2)\|
\end{equation*}
as well as 
\begin{equation*}
\begin{aligned}
& \int_{I^{n,m}_1}\Bigg|\Bigg|f_1(t^{n,m}_1, u_1(t^{n,m}_1), u_2(t^{n,m}_1)) - f_1(t^{n,m}_1, u_1(t^{n,m}_1), \bar{I}^k u_2)\Bigg|\Bigg| \diff t \\
& \quad \leq  L_1 \frac{k^{n,m}_1}{k^n}\left|\left| \int_{I^{n}} \left(u_2(t^{n,m}_1) -  u_2(s) \right) \diff s\right|\right| \leq L_1 k^{n,m}_1 k^n\max_{t \in I} \|f_1(t, u_1, u_2)\|.
\end{aligned}
\end{equation*}
Analogously, one can analyze $e_2^k$. Once we do that, we can proceed with Gronwall's lemma. The terms corresponding to the last time-step in the inequality (\ref{gronwall_equation}) are then transferred from the right to the left side. If we assume that $(k^N_1 + k^N_2)(L_1  + L_2) \leq \frac{1}{2}$, then it holds

\begin{equation*}
  \frac{1}{2} \left(\|e_1^k(t^N)\| +\|e_2^k(t^N)\| \right)
  \leq   \left\{1 - k_1^N(L_1 + L_2)\right\}\|e_1^k(t^N)\| + \left\{1 - k_2^N(L_1 + L_2)\right\}\|e_2^k(t^N)\|.
\end{equation*}
Applying Gronwall's lemma yields the results.
\end{proof}

\begin{remark}[Separation of the time scales]
  \label{remark:sep}
  In a simplified form, we were able to prove
  \begin{equation*}
    \begin{aligned}
      \big| \big| \textbf{e}^k(t^N)\big|\big|  & = \mathcal{O}\left(k_1\left| \left| d_t f_1(t, u_1, u_2) \right| \right|\right) + \mathcal{O}\left(k \left| \left| f_1(t, u_1, u_2) \right| \right| \right) \\
      &\quad + \mathcal{O}\left(k_2\left| \left| d_t f_2(t, u_1, u_2) \right| \right|\right) + \mathcal{O}\left(k \left| \left| f_2(t, u_1, u_2) \right| \right| \right).
    \end{aligned}
  \end{equation*}
  Based on that, we can make a few observations. First, we obtained linear convergence in time typical for the implicit Euler scheme.
  Second, while we were not able to fully decouple the two problems, the macro time step $k$ only acts on the lower order term $f$ but not its derivative. This indicates that the oscillations of functions $f_1$, and $f_2$ are localizable to each of the two subproblems and an efficient discretization by a multirate method is possible. This is in agreement with the a posteriori error estimate and the numerical results demonstrated in~\cite{Soszynska2021}. 
\end{remark}


\section{Coupling of Heat Equations}
 In this section, we consider a heat equation prescribed on two domains $\bar{\Omega}_1 \cup \bar{\Omega}_2 = \Omega \subset \mathbb{R}^d $ for $d \in \{ 2, 3\}$ with a common interface $\Gamma$. The domains are illustrated in Figure~\ref{domain}. On each of the domains, we choose different diffusivity constants $\boldsymbol{\nu} = (\nu_1, \nu_2)^T$ and external forces $\textbf{f} = (\textbf{f}_1, \textbf{f}_2)^T \in L^2(\Omega)^d$. We define a space-time domain for any arbitrary function space $V$
\begin{equation}
    X(V) \coloneqq \left\{ v \in L^2(I, V) |\; \partial_t v \in L^2(I, V^*) \right\},
\label{space_time_domain}
\end{equation}
and take $\textbf{u} = (\textbf{u}_1, \textbf{u}_2)^T$ with $\textbf{u}_j \in X(H^1(\Omega_j))^d$ for $j = 1,2$. The solution $\textbf{u}: \Omega \times I \to \mathbb{R}^d$ is governed by the equations
\begin{equation}
\partial_t \textbf{u}_1 - \nu_1 \Delta \textbf{u}_1 = \textbf{f}_1\;\text{ in }\Omega_1\text{ and }
\partial_t \textbf{u}_2 - \nu_2 \Delta \textbf{u}_2 = \textbf{f}_2\;\text{ in }\Omega_2. 
\label{heat_continuous}
\end{equation}
\begin{figure}
\centering
\begin{tikzpicture}[scale = 0.6]
    \draw  plot[smooth, tension=.7, very thick] coordinates {(-3.5,0.5) (-3,2.5) (-1,3.5) (1.5,3) (4,3.5) (5,2.5) (5,0.5) (2.5,-2) (0,-0.5) (-3,-2) (-3.5,0.5)};
    \draw (-5.0, -3.0) -- (6.5, -3.0) -- (6.5, 5.0) -- (-5.0, 5.0) -- (-5.0, -3.0);
    \node at (5.5, -2.0) {\large{$\Omega_2$}};
    \node at (0.75, 1.0) {\large{$\Omega_1$}};
    \node at (0.75, 1.0) {\large{$\Omega_1$}};
    \node at (-3.0, 3.3) {\large{$\Gamma$}};
\end{tikzpicture}
\caption{We show the splitting of the domain $\Omega$ into $\Omega_1$ and $\Omega_2$ with a common interface $\Gamma$.}
\label{domain}
\end{figure}
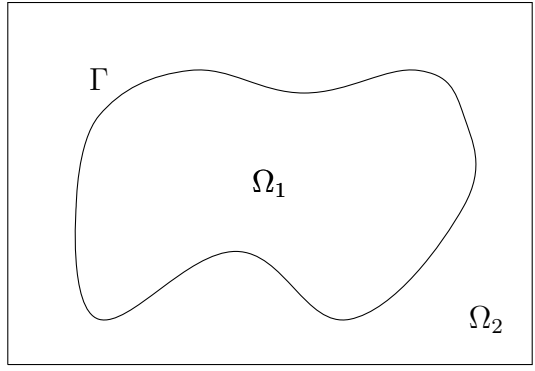
 On the interface $\Gamma=\partial\Omega_1\cap\partial\Omega_2$ we impose coupling conditions typical for continuous two-phase flow problems, see, for instance~\cite{Gross2006}, that is continuity of the solutions (kinematic condition) and balance of stress in the normal direction (dynamic condition) 
\begin{equation}
\textbf{u}_1 = \textbf{u}_2,\quad 
\nu_1 \partial_{\textbf{n}_1}\textbf{u}_1 = -\nu_2 \partial_{\textbf{n}_2}\textbf{u}_2 \text{ on }\Gamma.
\label{coupling_conditions}
\end{equation}
By $\textbf{n}_1$ and $\textbf{n}_2$ we denote normal vectors corresponding to each of the domains. In particular, on the interface we have $\textbf{n}_1 = - \textbf{n}_2$.  On the outer boundary $\partial \Omega$ we choose a no-slip boundary condition $\textbf{u}_1 = \textbf{u}_2 = \textbf{0}$. Similarly, at initial time we set $\textbf{u}_1(0) = \textbf{u}_2(0) = \textbf{0}$.  After integration by parts using test functions $\boldsymbol{\varphi} = (\boldsymbol{\varphi}_1, \boldsymbol{\varphi}_2)^T$ with $\boldsymbol{\varphi}_j \in X(H^1(\Omega_j))^d$ for $j = 1,2$,  we obtain
\begin{equation}
  \int_{I} \Big\{\left(\textbf{d}_t \textbf{u}, \boldsymbol{\varphi}\right)_{\Omega} + \boldsymbol{\nu}\left(\nabla \textbf{u}, \nabla \boldsymbol{\varphi}\right)_{\Omega} 
  -  \left\langle \nu_1 \partial_{\textbf{n}_1} \textbf{u}_1, \boldsymbol{\varphi}_1\right\rangle_{\Gamma} -  \left\langle \nu_2 \partial_{\textbf{n}_2} \textbf{u}_2, \boldsymbol{\varphi}_2 \right\rangle_{\Gamma}  \Big\}\diff t 
  = \int_{I} \left(\textbf{f}, \boldsymbol{\varphi} \right)_{\Omega} \diff t.
  \label{variational_heat_continuous}
\end{equation}
Given the coupling conditions (\ref{coupling_conditions}),
we have
\begin{equation}
\nu_1 \partial_{\textbf{n}_1} \textbf{u}_1 = \frac{1}{2} \left(\nu_1 \partial_{\textbf{n}_1} \textbf{u}_1 - \nu_2 \partial_{\textbf{n}_2} \textbf{u}_2 \right) = -\nu_2 \partial_{\textbf{n}_2} \textbf{u}_2
\label{neumann_coupling}
\end{equation}
and therefore the interface terms are equal to
\begin{equation*}
    -  \left\langle \nu_1 \partial_{\textbf{n}_1} \textbf{u}_1, \boldsymbol{\varphi}_1\right\rangle_{\Gamma} -  \left\langle \nu_2 \partial_{\textbf{n}_2} \textbf{u}_2, \boldsymbol{\varphi}_2 \right\rangle_{\Gamma}  = \frac{1}{2}\left\langle \nu_1 \partial_{\textbf{n}_1} \textbf{u}_1 - \nu_2 \partial_{\textbf{n}_2} \textbf{u}_2, \boldsymbol{\varphi}_2 - \boldsymbol{\varphi}_1\right\rangle_{\Gamma}.
\end{equation*}
To symmetrize the formulation, we subtract
\begin{equation}
     \frac{1}{2} \left\langle \textbf{u}_2 - \textbf{u}_1, \nu_1 \partial_{\textbf{n}_1} \boldsymbol{\varphi}_1 - \nu_2 \partial_{\textbf{n}_2} \boldsymbol{\varphi}_2\right\rangle_{\Gamma},
    \label{symmetric_coupling}
\end{equation}
which vanishes once $\textbf{u}_1=\textbf{u}_2$ on the interface. Further we add a Nitsche term
\begin{equation*}
\gamma\left\langle \textbf{u}_2 - \textbf{u}_1, \boldsymbol{\varphi}_2 - \boldsymbol{\varphi}_1\right\rangle_{\Gamma}.
\end{equation*}
By $\gamma$ we denote the Nitsche constant, we refer to the original paper of NItsche~\cite{nitsche} and to~\cite{BurmanFernandez}, where similar approaches are applied to two-phase flow problems and fluid-structure interactions. Since the exact solution fulfills coupling conditions and therefore the interface terms are equal to zero, we obtain a consistent and, as we will later see, coercive formulation
\begin{multline}
  a(\textbf{u}, \boldsymbol{\varphi})  \coloneqq  \int_{I} \Big\{\left(\textbf{d}_t \textbf{u}, \boldsymbol{\varphi}\right)_{\Omega} + \boldsymbol{\nu}\left(\nabla \textbf{u}, \nabla \boldsymbol{\varphi}\right)_{\Omega}
    + \gamma\left\langle \textbf{u}_2 - \textbf{u}_1, \boldsymbol{\varphi}_2 - \boldsymbol{\varphi}_1\right\rangle_{\Gamma} 
    \\
    \quad +  \frac{1}{2}\left\langle \nu_1 \partial_{\textbf{n}_1} \textbf{u}_1 - \nu_2 \partial_{\textbf{n}_2} \textbf{u}_2, \boldsymbol{\varphi}_2 - \boldsymbol{\varphi}_1\right\rangle_{\Gamma} 
    - \frac{1}{2} \left\langle \textbf{u}_2 - \textbf{u}_1, \nu_1 \partial_{\textbf{n}_1} \boldsymbol{\varphi}_1 - \nu_2 \partial_{\textbf{n}_2} \boldsymbol{\varphi}_2\right\rangle_{\Gamma} 
    \Big\}\diff t\\
    = \int_{I} \left(\textbf{f}, \boldsymbol{\varphi} \right)_{\Omega} \diff t.
  \label{variational_heat_continuous_symmetric}
\end{multline}
This formulation guarantees the fulfillment of the coupling conditions~(\ref{coupling_conditions}) even without any assumptions on the continuity of trial and test functions on the interface. Due to the arbitrariness of test functions, condition~(\ref{symmetric_coupling}) leads to the continuity across the interface of solutions. Additional interface terms coming from integration by parts and returning to the strong formulation~(\ref{heat_continuous}) as well as property~(\ref{neumann_coupling}) guarantee the balance of stress.
This variational treatment of the interface conditions was first proposed by P. Hansbo and M. G. Larson in \cite{DG}.

 We will use the notations $(\cdot, \cdot)_{\Omega_1}$, $(\cdot, \cdot)_{\Omega_2}$ and $(\cdot, \cdot)_{\Omega}$ to indicate the $L^2$-product over a corresponding domain. We will denote the norms over each of the domains in a similar way. On the interface, using Riesz representation theorem, we then define
\begin{equation*}
    \langle \textbf{u}, \boldsymbol{\varphi} \rangle_{\Gamma} \coloneqq \langle \textbf{u}, \boldsymbol{\varphi} \rangle_{H^{-\frac{1}{2}}(\Gamma)^d \times H^{\frac{1}{2}} (\Gamma)^d}\ , \hspace{0.5 cm} ||\textbf{u}||_{\Gamma} \coloneqq \sqrt{\langle \textbf{u}, \textbf{u}\rangle_{\Gamma}} \ .
\end{equation*}


\subsection{Discretization in Time}
 With the help of the projection operators $\textbf{I}_1^k$ and $\textbf{I}_2^k$ (we use bold letters to indicate that we perform projections on multidimensional functions), we are ready to formulate a semi-discrete variational problem again using the implicit Euler time-stepping scheme
\begin{equation}
  \begin{aligned}
    a^{k}(\textbf{u}^{k}, &\boldsymbol{\varphi}^{k})  \coloneqq  \int_{I} \Big\{\big(\textbf{d}_t^k \textbf{u}^{k}, \boldsymbol{\varphi}^{k}\big)_{\Omega} + \boldsymbol{\nu}\left(\nabla \textbf{u}^{k}, \nabla \boldsymbol{\varphi}^{k}\right)_{\Omega} \\
    &\quad - \frac{1}{2} \left\langle \nu_1 \partial_{\textbf{n}_1} \textbf{u}_1^k - \nu_2 \partial_{\textbf{n}_2} \textbf{I}_1^k \textbf{u}_2^k, \boldsymbol{\varphi}_1^k \right\rangle_{\Gamma} 
    + \frac{1}{2} \left\langle \nu_1 \partial_{\textbf{n}_1} \textbf{I}_2^k \textbf{u}_1^k - \nu_2 \partial_{\textbf{n}_2} \textbf{u}_2^k, \boldsymbol{\varphi}_2^k \right\rangle_{\Gamma}
    \\
    & \quad - \frac{1}{2} \left\langle \textbf{I}_1^k \textbf{u}_2^k - \textbf{u}_1^k, \nu_1 \partial_{\textbf{n}_1}\boldsymbol{\varphi}_1^k \right\rangle_{\Gamma}
    + \frac{1}{2} \left\langle \textbf{u}_2^k - \textbf{I}_2^k\textbf{u}_1^k, \nu_2 \partial_{\textbf{n}_2}\boldsymbol{\varphi}_2^k \right\rangle_{\Gamma}
    \\
    & \quad - \gamma \left\langle \textbf{I}_1^k \textbf{u}_2^k - \textbf{u}_1^k, \boldsymbol{\varphi}_1^k \right\rangle_{\Gamma}
    + \gamma \left\langle \textbf{u}_2^k -\textbf{I}_2^k\textbf{u}_1^k, \boldsymbol{\varphi}_2^k \right\rangle_{\Gamma}
    \Big\}\diff t
    = \int_{I} \left(\textbf{f}, \boldsymbol{\varphi}^{k} \right)_{\Omega} \diff t.
  \end{aligned}
  \label{discrete_heat}
\end{equation}
The corresponding function spaces are defined as ($j=1,2$) 
\begin{equation} 
    X^{k}_j \coloneqq{} \big\{ \boldsymbol{\varphi} \in L^2(\bar{I}, H^1(\Omega_{j}))\Big|\; \varphi|_{I^{n,m}_j} 
    \in {\cal P}_0(H^1({\Omega_j})) 
    \text{ for all }
    I^{n,m}_j \subset I, \; \boldsymbol{\varphi}(0) = \textbf{0} \big\}
\end{equation}
and $X^k \coloneqq X^k_1 \times X^k_2$. The test space $Y^k$ is defined in a similar fashion, however, using global $H^1$ functions over the whole domain $\Omega$ 
\begin{align*}
  Y^{k} \coloneqq {} \big\{\boldsymbol{\varphi} = (\varphi_1, \varphi_2)^T \in L^2(\bar{I},  H^1(\Omega))\Big| & \; \varphi_1|_{\Omega_1} \in X^{k}_1, \varphi_2|_{\Omega_2} \in X^{k}_2  \big\},
\end{align*}
Testing with $\boldsymbol{\varphi}^k \in (Y^{k})^d$ allows us to recover the coupling conditions in a weak form, see~\cite[Sec. 3.4]{Richter2017}
\begin{equation}
  \begin{aligned}
    &0 =  \int_{I_n} \big\langle \textbf{u}_2^k - \textbf{u}_1^k, \boldsymbol{\varphi}_1^k \big\rangle_{\Gamma} \diff t = \int_{I_n} \big\langle \textbf{u}_2^k - \textbf{u}_1^k, \boldsymbol{\varphi}_2^k \big\rangle_{\Gamma} \diff t, \\
    &0 = \int_{I_n} \big\langle\nu_1 \partial_{\textbf{n}_1} \textbf{u}_1^k + \nu_2 \partial_{\textbf{n}_2} \textbf{u}_2^k, \boldsymbol{\varphi}_1^k \rangle_{\Gamma} \diff t= \int_{I_n} \big\langle\nu_1 \partial_{\textbf{n}_1} \textbf{u}_1^k + \nu_2 \partial_{\textbf{n}_2} \textbf{u}_2^k, \boldsymbol{\varphi}_2^k \rangle_{\Gamma} \diff t.
  \end{aligned}
  \label{weak_coupling_heat}
\end{equation}

\begin{theorem}
Let $\textbf{u} \in X(H^1_0(\Omega))^d$, $\textbf{u}_j \in W^{1, \infty}(H^2(\Omega_j))^d$ for $j=1,2$ be continuous solutions to~(\ref{variational_heat_continuous_symmetric}) and $\textbf{u}^k~\in~(X^k)^d$ their semi-discrete counterpart and a solution to~(\ref{discrete_heat}), then the following estimate holds
\begin{equation}
	\big| \big| \textbf{e}^k(t^N)\big|\big|^2_{\Omega} + \int_I \boldsymbol{\nu}^2\left|\left|\nabla \textbf{e}^k \right|\right|^2_{\Omega}   \diff t
   \leq
  \sum_{j=1}^2
  C_j\sum_{n = 1}^{N}\sum_{m = 1}^{N^n_j}  \Bigg\{(k^{n, m}_j)^3 \max_{t \in I} \left| \left| \textbf{d}_t \nabla \textbf{u}_j \right| \right|^2_{\Omega_j}  + (k^{n, m}_1)^3 \max_{t \in I} \left|\left|\textbf{d}_t \partial_{\textbf{n}_j}\textbf{u}_j\right|\right|_{\Gamma}^2\Bigg\} 
\end{equation}
where the errors $\textbf{e}^k = (\textbf{e}_1^k, \textbf{e}_2^k)^T$ are defined as $\textbf{e}_j^k \coloneqq \textbf{u}_j^k - \textbf{i}_j^k \textbf{u}_j$, for $j=1,2$.
\label{theorem_semidiscrete_heat}
\end{theorem}
\begin{proof}
Using Galerkin orthogonality, we have $a^k(\textbf{u}^k, \textbf{e}^k) = a(\textbf{u}, \textbf{e}^k)$ and therefore it holds
\begin{equation}
    a^k(\textbf{u}^k, \textbf{e}^k) - \int_{I} \Big\{\big(\textbf{d}_t^k \textbf{i}^k \textbf{u}, \textbf{e}^k\big)_{\Omega} + \boldsymbol{\nu} \big(\nabla \textbf{i}^k \textbf{u}, \nabla{\textbf{e}^k} \big)_{\Omega} \Big\} \diff = a(\textbf{u}, \textbf{e}^k) - \int_{I} \Big\{\big(\textbf{d}_t^k \textbf{i}^k \textbf{u}, \textbf{e}^k\big)_{\Omega} + \boldsymbol{\nu} \big(\nabla \textbf{i}^k \textbf{u}, \nabla{\textbf{e}^k} \big)_{\Omega} \Big\} \diff t.
\label{semi_discrete_galerkin}
\end{equation}
By adding and subtracting terms, the lest side of this identity can be rewritten as 
\begin{equation}
    \begin{aligned}
      & a^k(\textbf{u}^k, \textbf{e}^k) - \int_{I} \Big\{\big(\textbf{d}_t^k \textbf{i}^k \textbf{u}, \textbf{e}^k\big)_{\Omega} + \boldsymbol{\nu} \big(\nabla \textbf{i}^k \textbf{u}, \nabla{\textbf{e}^k} \big)_{\Omega} \Big\} \diff t \\
      & \qquad= \int_{I} \Big\{ \Big[\big(\textbf{d}_t^k \textbf{e}^k, \textbf{e}^k \big)_{\Omega} + \boldsymbol{\nu} \big|\big|\nabla \textbf{e}^k \big|\big|_{\Omega}^2\Big]^{(iv)}
      + \Big[\frac{1}{2}\left\langle \nu_1 \partial_{\textbf{n}_1} \textbf{u}_1^k - \nu_2 \partial_{\textbf{n}_2} \textbf{u}_2^k, \textbf{e}_2^k - \textbf{e}_1^k\right\rangle_{\Gamma}\Big]^{(iii)} \\
      & \qquad \Big[- \frac{1}{2}\left\langle \textbf{u}_2^k - \textbf{u}_1^k, \nu_1 \partial_{\textbf{n}_1} \textbf{e}_1^k - \nu_2 \partial_{\textbf{n}_2} \textbf{e}_2^k\right\rangle_{\Gamma} +\gamma \left\langle \textbf{u}_2^k - \textbf{u}_1^k, \textbf{e}_2^k - \textbf{e}_1^k \right\rangle_{\Gamma}\Big]^{(i)} \\
      & \qquad
      \Big[- \frac{1}{2} \big\langle  \nu_2 \partial_{\textbf{n}_2} (\textbf{u}_2^{k} - \textbf{I}_1^k \textbf{u}_2^{k}), \textbf{e}_1^{k} \big\rangle_{\Gamma} 
+ \frac{1}{2} \big\langle \nu_1 \partial_{\textbf{n}_1}( \textbf{I}_2^k \textbf{u}_1^{k} - \textbf{u}_1^{k}), \textbf{e}_2^{k} \big\rangle_{\Gamma}
\\
& \qquad- \frac{1}{2} \big\langle \textbf{I}_1^k \textbf{u}_2^{k} - \textbf{u}_2^{k}, \nu_1 \partial_{\textbf{n}_1}\textbf{e}_1^{k} \big\rangle_{\Gamma} 
+ \frac{1}{2} \big\langle \textbf{u}_1^{k} - \textbf{I}_2^k\textbf{u}_1^{k}, \nu_2 \partial_{\textbf{n}_2}\textbf{e}_2^{k} \big\rangle_{\Gamma}
 - \gamma \big\langle \textbf{I}_1^k \textbf{u}_2^{k} - \textbf{u}_2^{k}, \textbf{e}_1^{k} \big\rangle_{\Gamma} \\
&\qquad + \gamma \big\langle \textbf{u}_1^{k} - \textbf{I}_2^k\textbf{u}_1^{k}, \textbf{e}_2^{k} \big\rangle_{\Gamma}\Big]^{(ii)}
    \Big\} \diff t. 
    \end{aligned}
    \label{heat_orthogonal}
\end{equation}
Since the semi-discrete solution fulfills the coupling conditions in the sense of~(\ref{weak_coupling_heat}), specifically the Dirichlet condition, we have for term $(i)$
\begin{equation*}
\int_I \left\{- \frac{1}{2}\left\langle \textbf{u}_2^k - \textbf{u}_1^k, \nu_1 \partial_{\textbf{n}_1} \textbf{e}_1^k - \nu_2 \partial_{\textbf{n}_2} \textbf{e}_2^k\right\rangle_{\Gamma}+\gamma \left\langle \textbf{u}_2^k - \textbf{u}_1^k, \textbf{e}_2^k - \textbf{e}_1^k \right\rangle_{\Gamma} \right\} \diff t = 0.
\end{equation*}

We can further simplify expression~(\ref{heat_orthogonal}) by noticing, that on every macro time-step $I^n$ the term $(ii)$ vanishes
\begin{equation}
    \begin{aligned}
    &\int_{I^n} \Big\{ - \frac{1}{2} \big\langle  \nu_2 \partial_{\textbf{n}_2} (\textbf{u}_2^{k} - \textbf{I}_1^k \textbf{u}_2^{k}), \textbf{e}_1^{k} \big\rangle_{\Gamma} 
+ \frac{1}{2} \big\langle \nu_1 \partial_{\textbf{n}_1}( \textbf{I}_2^k \textbf{u}_1^{k} - \textbf{u}_1^{k}), \textbf{e}_2^{k} \big\rangle_{\Gamma}
\\ & \qquad - \frac{1}{2} \big\langle \textbf{I}_1^k \textbf{u}_2^{k} - \textbf{u}_2^{k}, \nu_1 \partial_{\textbf{n}_1}\textbf{e}_1^{k} \big\rangle_{\Gamma}
+ \frac{1}{2} \big\langle \textbf{u}_1^{k} - \textbf{I}_2^k\textbf{u}_1^{k}, \nu_2 \partial_{\textbf{n}_2}\textbf{e}_2^{k} \big\rangle_{\Gamma}
 - \gamma \big\langle \textbf{I}_1^k \textbf{u}_2^{k} - \textbf{u}_2^{k}, \textbf{e}_1^{k} \big\rangle_{\Gamma}\\
 &\qquad + \gamma \big\langle \textbf{u}_1^{k} - \textbf{I}_2^k\textbf{u}_1^{k}, \textbf{e}_2^{k} \big\rangle_{\Gamma} \Big\} \diff t = 0. 
    \end{aligned}
\label{interface_zero}
\end{equation}
To explain that, let us look at the integral 
$
    \int_{I^n} \big\langle \textbf{u}_1^{k} - I_2^k\textbf{u}_1^{k}, \textbf{e}_2^{k} \big\rangle_{\Gamma} \diff t.
$
    From the construction of our time meshes, we have two possibilities. According to the first one, there is no micro time-stepping in the domain $\Omega_1$, in other words, $N^n_1 = 1$ and $\textbf{I}_1^k \big|_{I^n} = \bar{\textbf{I}}^k \big|_{I^n}$. Then $\textbf{I}_2^k\textbf{u}_1^{k} = \textbf{u}_1^{k}$ and therefore $ \big\langle \textbf{u}_1^{k} - \textbf{I}_2^k\textbf{u}_1^{k}, \textbf{e}_2^{k} \big\rangle_{\Gamma} = 0$. Otherwise, we have no micro time-stepping in the domain $\Omega_2$ ($N^n_2 = 1$ and $\textbf{I}_2^k \big|_{I^n} = \bar{\textbf{I}}^k \big|_{I^n}$). In this case, knowing that the test function $\textbf{e}_2^k$ is a continuous constant over the interval $I^n$ and using the property (\ref{average_zero}) of the projection operator $\textbf{I}_2^k$, we can write
\begin{equation*}
    \int_{I^n} \big\langle \textbf{u}_1^{k} - \textbf{I}_2^k\textbf{u}_1^{k}, \textbf{e}_2^{k} \big\rangle_{\Gamma} \diff t = \left\langle \int_{I^n} (\textbf{u}_1^{k} - \textbf{I}_2^k\textbf{u}_1^{k}) \diff t , \textbf{e}_2^{k} \right\rangle_{\Gamma} = 0.
\end{equation*}
The reasoning corresponding to the remaining terms in~(\ref{interface_zero}) is analogous. Let us look at the unresolved interface term.

Next, we exploit the weak coupling conditions and transform \emph{(iii)} as
\begin{equation*}
     \int_{I^n} \frac{1}{2}\left\langle \nu_1 \partial_{\textbf{n}_1} \textbf{u}_1^k - \nu_2 \partial_{\textbf{n}_2} \textbf{u}_2^k, \textbf{e}_2^k - \textbf{e}_1^k\right\rangle_{\Gamma} \diff t= \int_{I^n} \frac{1}{2}\left\langle \nu_1 \partial_{\textbf{n}_1} \textbf{u}_1^k - \nu_2 \partial_{\textbf{n}_2} \textbf{u}_2^k, \textbf{i}_1^k \textbf{u}_1 - \textbf{i}_2^k \textbf{u}_2 \right\rangle_{\Gamma} \diff t 
\end{equation*}
We will again look closely at the implications of our time mesh structure. Because of the symmetry of this expression, without loss of generality, we can assume that $N^n_1 = 1$ and $\textbf{I}_1^k \big|_{I^n} = \bar{\textbf{I}}^k \big|_{I^n}$. The use of both the weak~(\ref{weak_coupling_heat}) and the strong~(\ref{coupling_conditions}) coupling conditions leads us to 
\begin{multline*}
  \int_{I^n} \frac{1}{2}\left\langle \nu_1 \partial_{\textbf{n}_1} \textbf{u}_1^k - \nu_2 \partial_{\textbf{n}_2} \textbf{u}_2^k, \textbf{i}_1^k \textbf{u}_1 - \textbf{i}_2^k \textbf{u}_2 \right\rangle_{\Gamma} \diff t\\
  = \left\langle \nu_1 \partial_{\textbf{n}_1} \textbf{u}_1^k ,\int_{I^n}( \textbf{i}_1^k \textbf{u}_1 - \textbf{u}_1) \diff t \right\rangle_{\Gamma} + \left\langle \nu_1 \partial_{\textbf{n}_1} \textbf{u}_1^k , \int_{I^n}(\textbf{u}_2 - \textbf{i}_2^k \textbf{u}_2 ) \diff t \right\rangle_{\Gamma} = 0.
\end{multline*}
The left side of the identity~(\ref{heat_orthogonal}), with the help of the strong coupling conditions, is equal to 
\begin{multline*}
    a(\textbf{u}, \textbf{e}^k) - \int_{I} \Big\{\big(\textbf{d}_t^k \textbf{i}^k \textbf{u}, \textbf{e}^k\big)_{\Omega} + \boldsymbol{\nu} \big(\nabla \textbf{i}^k \textbf{u}, \nabla{\textbf{e}^k} \big)_{\Omega} \Big\} \diff t \\ 
    = \int_{I} \Big\{\big(\textbf{d}_t \textbf{u} - \textbf{d}_t^k \textbf{i}^k \textbf{u}, \textbf{e}^k \big)_{\Omega} + \boldsymbol{\nu} \big(\nabla (\textbf{u} -  \textbf{i}^k \textbf{u}), \nabla \textbf{e}^k \big)_{\Omega}  + \frac{1}{2}\left\langle \nu_1 \partial_{\textbf{n}_1} \textbf{u}_1 - \nu_2 \partial_{\textbf{n}_2} \textbf{u}_2, \textbf{e}_2^k - \textbf{e}_1^k\right\rangle_{\Gamma} \Big\} \diff t
\end{multline*}
While analyzing the remaining terms~\emph{(iv)} in~(\ref{heat_orthogonal}), we will concentrate on the error $\textbf{e}_1^k$ since the estimations corresponding to the second error are very similar. Therefore, we will consider a single interval $I^{n,m}_1$. Let us start with the time discretization error

\begin{equation*}
    \int_{I^{n,m}_1} \textbf{d}_t \textbf{u}_1 \diff t = \textbf{u}_1(t^{n,m}_1) - \textbf{u}_1(t^{n, m-1}_1) = \int_{I^{n,m}_1} \textbf{d}_t^k \textbf{i}_1^k \textbf{u}_1 \diff t.
\end{equation*}
As a result, we have
\begin{equation}
    \int_{I^{n,m}_1} \big(\textbf{d}_t \textbf{u} - \textbf{d}_t^k \textbf{i}_1^k \textbf{u}, \textbf{e}^k \big)_{\Omega} \diff t = \textbf{0}.
    \label{time_equivalence}
\end{equation}

We will now examine the Laplacian terms~\emph{(iv)} on the right side of~(\ref{heat_orthogonal}). Knowing that the error $\textbf{e}_1^k$ is constant in time on every interval $I^{n,m}_1$, we have
\begin{multline}
     \left| \int_{I^{n, m}_1} \nu_1 \left(\nabla (\textbf{u}_1- \textbf{i}_1^k \textbf{u}_1), \nabla \textbf{e}_1^k \right)_{\Omega_1} \diff t \right|
      =    \left| \int_{I^{n, m}_1} \int_{t}^{t^{n,m}_1} \nu_1 \left(\textbf{d}_t \nabla \textbf{u}_1(s), \nabla \textbf{e}_1^k(t^{n,m}_1) \right)_{\Omega_1} \diff s \diff t \right| \\
      \leq
 c(k^{n, m}_1)^3\max_{t \in I}  \left| \left|\textbf{d}_t \nabla \textbf{u}_1 \right| \right|^2_{\Omega_1} + \frac{1}{8} \nu_1^2 \int_{I^{n, m}_1} \left| \left|
     \nabla \textbf{e}_1^k(t^{n, m}_1)\right|\right|^2_{\Omega_1} \diff t.
    \label{laplace_time}
\end{multline}
The remaining time discretization term in~(\ref{heat_orthogonal}) can be rewritten as
\begin{equation}
    \int_{I^{n,m}_1} \left(\textbf{d}_t^k \textbf{e}_1^k, \textbf{e}_1^k \right)_{\Omega_1} \diff t = \frac{1}{2} \left|\left|\textbf{e}_1^k(t^{n,m}_1)\right|\right|^2_{\Omega_1} - \frac{1}{2} \left|\left|\textbf{e}_1^k(t^{n , m - 1}_1)\right|\right|^2_{\Omega_1}  + \frac{1}{2}\left|\left|\textbf{e}_1^k(t^{n, m}_1) - \textbf{e}_1^k(t^{n,m - 1}_1)\right|\right|^2_{\Omega_1}.
\label{time_identity}
\end{equation}
Summing these terms over the whole time interval, we obtain
\begin{flalign*}
    \int_{I} \left(\textbf{d}_t^k \textbf{e}_1^k, \textbf{e}_1^k \right)_{\Omega_1} \diff t & = \frac{1}{2} \left|\left|\textbf{e}_1^k(t^{N})\right|\right|^2_{\Omega_1} + \frac{1}{2}\sum_{n = 1}^{N}\sum_{m = 1}^{N_1^n}\left|\left|\textbf{e}_1^k(t^{n, m}_1) - \textbf{e}_1^k(t^{n,m - 1}_1)\right|\right|^2_{\Omega_1}
\end{flalign*}
and therefore 
\begin{equation*}
      \big|\big|\textbf{e}^k(t^N)\big|\big|_{\Omega}^2 \leq \int_I 2\left(\textbf{d}_t^k \textbf{e}^k, \textbf{e}^k \right)_{\Omega} \diff t.
\end{equation*}
We proceed to the last interface term on the right side of~(\ref{semi_discrete_galerkin}). Implementing very similar solutions as in the analysis of the previous interface terms, we can show that 
\begin{equation}
\begin{aligned}
  \int_{I^n} \frac{1}{2}&\left\langle \nu_1 \partial_{\textbf{n}_1} \textbf{u}_1 - \nu_2 \partial_{\textbf{n}_2} \textbf{u}_2, \textbf{e}_2^k - \textbf{e}_1^k\right\rangle_{\Gamma} \diff t\\
  &= \int_{I^n} \frac{1}{2}\left\langle \nu_1 \partial_{\textbf{n}_1} (\textbf{u}_1 - \textbf{I}^k_1 \textbf{u}_1 ) - \nu_2 \partial_{\textbf{n}_2} (\textbf{u}_2 - \textbf{I}^k_2 \textbf{u}_2 ), \textbf{e}_2^k - \textbf{e}_1^k\right\rangle_{\Gamma} \diff t \\  
     & \leq \int_{I^n} \frac{1}{2} \left|\left| \nu_1 \partial_{\textbf{n}_1} (\textbf{u}_1 - \textbf{I}^k_1 \textbf{u}_1 ) - \nu_2 \partial_{\textbf{n}_2} (\textbf{u}_2 - \textbf{I}^k_2 \textbf{u}_2 )\right|\right|_{\Gamma}\left|\left|  \textbf{e}_2^k - \textbf{e}_1^k\right|\right|_{\Gamma} \diff t .
\end{aligned}
\label{neumann_heat}
\end{equation}

The first term in~(\ref{neumann_heat}) can be estimated using the fundamental theorem of calculus, from which follows
\begin{equation*}
    \begin{aligned}
    & \nu_j\left|\left|\partial_{\textbf{n}_j} (\textbf{u}_j - \textbf{I}^k_j \textbf{u}_j )\right|\right|_{\Gamma} \leq \nu_j k^{n,m}_j \max_{t \in I}\left|\left|\textbf{d}_t \partial_{\textbf{n}_j}\textbf{u}_j \right|\right|_{\Gamma}, \quad j=1,2.
    \end{aligned}
\end{equation*}
The other term can be dealt with by using the trace inequality
\begin{equation*}
    \begin{aligned}
    \left|\left| \textbf{e}_2^k - \textbf{e}_1^k\right|\right|_{\Gamma} \leq \left|\left| \textbf{e}_1^k\right|\right|_{\Gamma} + \left|\left| \textbf{e}_2^k\right|\right|_{\Gamma} \leq c_1\left|\left| \nabla \textbf{e}_1^k\right|\right|_{\Omega_1} + c_2\left|\left| \nabla \textbf{e}_2^k\right|\right|_{\Omega_2}.
    \end{aligned}
\end{equation*}
We can combine these estimates using the Young and Poincar\'e inequalities
\begin{multline}
  \left|\int_{I^n} \frac{1}{2}\left\langle \nu_1 \partial_{\textbf{n}_1} \textbf{u}_1 - \nu_2 \partial_{\textbf{n}_2} \textbf{u}_2, \textbf{e}_2^k - \textbf{e}_1^k\right\rangle_{\Gamma}\diff t \right| \\
        \leq \sum_{j=1}^2 \sum_{m = 1}^{N^n_j}\left\{ c_j  (k^{n, m}_j)^3  \max_{t \in I} \left|\left|\textbf{d}_t \partial_{\textbf{n}_j}\textbf{u}_j \right|\right|_{\Gamma}^2 + \frac{1}{8}\nu_j^2\int_{I^{n,m}_j}\left|\left| \nabla \textbf{e}_j^k(t^{n, m}_j) \right|\right|^2_{\Omega_1} \diff t\right\} 
    \label{interface_time}
\end{multline}
If we perform the same steps for the solution $\textbf{u}_2$ and sum these terms over the whole interval $I$ we will get the final result.
\end{proof}
In a more compact way, we just proved that
\begin{equation*}
    \begin{aligned}
    \big| \big| \textbf{e}^k(t^N)\big|\big| + \int_I \boldsymbol{\nu} \big| \big| \textbf{e}^k\big|\big| \diff t & = \mathcal{O}\left(k_1\left| \left| \textbf{d}_t \nabla \textbf{u}_1 \right| \right|_{\Omega_1} \right) + \mathcal{O}\left(k_1 \left|\left|\textbf{d}_t \partial_{\textbf{n}_1}\textbf{u}_1\right|\right|_{\Gamma} \right) \\ & + \mathcal{O}\left(k_2\left| \left| \textbf{d}_t \nabla \textbf{u}_2 \right| \right|_{\Omega_2}\right) + \mathcal{O}\left(k_2 \left|\left|\textbf{d}_t \partial_{\textbf{n}_2}\textbf{u}_2\right|\right|_{\Gamma} \right).
    \end{aligned}
\end{equation*}
The convergence is linear which is expected for the implicit Euler method. Here, we were able to fully decouple the system. Due to the interface coupling conditions, we were able to avoid interdependencies between contributions from different time discretizations. 

\subsection{Discretization in Time and Space}

 To discretize the problem in space, we introduce regular triangulations $\mathcal{T}^h_1$ and $\mathcal{T}^h_2$. We assume that they match across the interface $\Gamma$. By $K_1$ we denote an element of the mesh $\mathcal{T}^h_1$ and by $K_2$ an element of $\mathcal{T}^h_2$. Their sizes are denoted by $h_1^K$ and $h_2^K$, respectively. Further, 
\begin{equation*}
    h_1 \coloneqq \max_{K_1 \in \mathcal{T}^h_1} h_1^K, \hspace*{0.5 cm} h_2 \coloneqq \max_{K_2 \in \mathcal{T}^h_2} h_2^K, \hspace*{0.5 cm} h \coloneqq \max \left\{ h_1, h_2\right\}.
\end{equation*}
As function spaces, we take the space of continuous polynomials of order $r$ for $j=1,2$
\begin{equation*}
  X^{k, h}_j(r) ={} \big\{ \boldsymbol{\varphi} \in X_j^k\Big|\; \varphi|_{K_j} 
  \in {\cal P}_r({\Omega_j}) \text{ for all }
  K_j \in \mathcal{T}^h_j \text{ and } \varphi|_{\partial \Omega_1 \backslash \Gamma}  = \textbf{0} \big\},
\end{equation*}
and introduce $X^{k,h}(r) = X^{k,h}_1(r)\times X^{k,h}_2(r)$. We similarly define the test space
\begin{align*}
    Y^{k, h}(r) \coloneqq {} \big\{ & \boldsymbol{\varphi} \in Y^k\Big|\; \varphi|_{\Omega_1} 
      \in X^{k,h}_1(r) \text{ and } \varphi|_{\Omega_2} 
      \in X^{k, h}_2(r) \big\}.
\end{align*}
We will introduce a Ritz projection operator. To ensure continuity over the interface, we will define it over the space $Y^{k, h}(r)$ instead of $X^{k,h}(r)$. As a consequence, we take $$\textbf{R}^h \textbf{u} = (\textbf{R}^h_1 \textbf{u}_1, \textbf{R}^h_2 \textbf{u}_2)^T \in (Y^{k,h}(r))^d$$ defined by
\begin{equation}
    \big(\nabla \textbf{R}^h \textbf{u}, \nabla \boldsymbol{\varphi}^{k,h}\big)_{\Omega} = \big(\nabla \textbf{u}, \nabla \boldsymbol{\varphi}^{k,h}\big)_{\Omega} \hspace*{0.3 cm}
    \textnormal{ for all } \hspace*{0.3 cm} \boldsymbol{\varphi}^{k,h} \in (Y^{k,h}(r))^d.
    \label{ritz}
\end{equation}
Again, since the Ritz projection operator is imposed on a function space consisting only of functions continuous across the interface, we have
\begin{equation}
    \textbf{R}_1^h\textbf{u}_1 \Big|_{\Gamma} = \textbf{R}_2^h\textbf{u}_2 \Big|_{\Gamma}.
    \label{ritz_continuity}
\end{equation}
We will now list some of the useful properties of the Ritz operator.

\begin{corollary}
  Given $\textbf{u} \in X(H^1_0(\Omega))^d$, $\textbf{u}_j^k \in X(H^{r +1}(\Omega_j))^d$ for $j=1,2$, the Ritz projection operator defined by~(\ref{ritz}) has the following properties:
   \label{ritz_properties} 
  \begin{enumerate}[label=(\roman*)]
\item
 \begin{equation} 
    \big| \big| \nabla^k(\textbf{u} - \textbf{R}^h \textbf{u})\big|\big|_{\Omega}  \leq c h^{r+1-k}\sum_{j=1}^2 \big|\big| \nabla^{r+1} \textbf{u}_j\big|\big|_{\Omega_j} \label{ritz_l2}\text{ for }k=0,1.
   \end{equation}
  \item 
  \begin{equation}
    \sum_{j=1}^2
    \big| \big|\nabla( \textbf{u}_j - \textbf{R}_j^h \textbf{u}_j) \cdot \textbf{n}_j\big|\big|_{\Gamma}
    \leq c \sum_{j=1}^2 h^{r - \frac{1}{2}} \big|\big| \nabla^{r+1} \textbf{u}_j\big|\big|_{\Omega_j}  \label{ritz_interface}
  \end{equation}
  \end{enumerate}
\end{corollary}

Then, our variational problem is given by 
\begin{equation}
\begin{aligned}
& a^{k, h}(\textbf{u}^{k,h}, \boldsymbol{\varphi}^{k,h}) =  \int_{I} \Big\{\big(\textbf{d}_t^k \textbf{u}^{k,h}, \boldsymbol{\varphi}^{k,h}\big)_{\Omega} + \boldsymbol{\nu}\left(\nabla \textbf{u}^{k,h}, \nabla \boldsymbol{\varphi}^{k,h}\right)_{\Omega} \\ & \quad 
- \frac{1}{2} \left\langle \nu_1 \partial_{\textbf{n}_1} \textbf{u}_1^{k,h} - \nu_2 \partial_{\textbf{n}_2} \textbf{I}_1^k \textbf{u}_2^{k,h}, \boldsymbol{\varphi}_1^{k,h} \right\rangle_{\Gamma} 
+ \frac{1}{2} \left\langle \nu_1 \partial_{\textbf{n}_1} \textbf{I}_2^k \textbf{u}_1^{k,h} - \nu_2 \partial_{\textbf{n}_2} \textbf{u}_2^{k,h}, \boldsymbol{\varphi}_2^{k,h} \right\rangle_{\Gamma}
\\ & \quad - \frac{1}{2} \left\langle \textbf{I}_1^k \textbf{u}_2^{k,h} - \textbf{u}_1^{k,h}, \nu_1 \partial_{\textbf{n}_1}\boldsymbol{\varphi}_1^{k,h} \right\rangle_{\Gamma}
 + \frac{1}{2} \left\langle \textbf{u}_2^{k,h} - \textbf{I}_2^k\textbf{u}_1^{k,h}, \nu_2 \partial_{\textbf{n}_2}\boldsymbol{\varphi}_2^{k,h} \right\rangle_{\Gamma}
\\ & \quad - \frac{\gamma}{h} \left\langle \textbf{I}_1^k \textbf{u}_2^{k,h} - \textbf{u}_1^{k,h}, \boldsymbol{\varphi}_1^{k,h} \right\rangle_{\Gamma}
  + \frac{\gamma}{h} \left\langle \textbf{u}_2^{k,h} -\textbf{I}_2^k\textbf{u}_1^{k,h}, \boldsymbol{\varphi}_2^{k,h} \right\rangle_{\Gamma}
\Big\}\diff t 
= \int_{I} \left(\textbf{f}, \boldsymbol{\varphi}^{k} \right)_{\Omega} \diff t.
\end{aligned}
\label{fully_discrete_heat}
\end{equation}
Moreover, we introduce a new norm
\begin{equation*}
    \left| \left| \left| \textbf{u} \right| \right| \right|_G \coloneqq \Big( \boldsymbol{\nu}^2\left|\left|\nabla \textbf{u}\right|\right|_G^2 + \frac{\gamma}{h} \left| \left| \textbf{u}_2 - \textbf{u}_1 \right| \right|_{\Gamma}^2 \Big)^{\frac{1}{2}}, 
\end{equation*}
where we can substitute $G$ with $\Omega_1$, $\Omega_2$, or $\Omega$ and use the appropriate component of $\boldsymbol{\nu}$.
We proceed with the error estimation for the fully discrete case.
\begin{theorem}
Let $\textbf{u} \in X(H^1_0(\Omega))^d$, $\textbf{u}_j \in W^{1, \infty}(H^{r + 1}(\Omega_j))^d$ for $j=1,2$ be continuous solutions to~(\ref{variational_heat_continuous_symmetric}) and $\textbf{u}^k~\in~(X^k)^d$ their discrete counterpart and a solution to~(\ref{fully_discrete_heat}), then the following estimate holds
%
\begin{multline*}
  \big| \big| \textbf{e}^{k,h}(t^N)\big|\big|^2_{\Omega}   + \int_I \big|\big|\big|\textbf{e}^{k,h} \big|\big|\big|^2_{\Omega} \diff t\\
  \leq C\sum_{j=1}^2\sum_{n = 1}^{N}\sum_{m = 1}^{N^n_j}  \Bigg\{(k^{n, m}_j)^3 \max_{t \in I} \left| \left| \textbf{d}_t \nabla \textbf{u}_j \right| \right|^2_{\Omega_j}   + (k^{n, m}_j)^3h \max_{t \in I} \left|\left|\textbf{d}_t \partial_{\textbf{n}_j}\textbf{u}_1\right|\right|_{\Gamma}^2
  \\
 + k^{n,m}_j  h^{2r + 2}  \max_{t \in I} \Big|\Big| \textbf{d}_t \nabla^{r + 1} \textbf{u}_j \Big| \Big|^2_{\Omega_j} + k^{n,m}_j h^{2r} \Big|\Big| \nabla^{r + 1} \textbf{u}_j(t^{n,m}_j) \Big| \Big|^2_{\Omega_j}
\Bigg\},
\end{multline*}
where the errors $\textbf{e}^{k,h} = (\textbf{e}_1^{k,h}, \textbf{e}_2^{k,h})^T$ are defined as
$\textbf{e}_j^{k,h} \coloneqq \textbf{u}_j^{k,h} - \textbf{i}_j^k \textbf{R}_j^h \textbf{u}_j$.
\label{theorem_discrete_heat}
\end{theorem}

\begin{proof}
We start using Galerkin orthogonality
\begin{equation}
    a^k(\textbf{u}^{k,h}, \textbf{e}^{k,h}) - a^k(\textbf{i}^k\textbf{R}^h\textbf{u}, \textbf{e}^{k,h}) = a(\textbf{u}, \textbf{e}^{k,h}) - a^k(\textbf{i}^k\textbf{R}^h\textbf{u}, \textbf{e}^{k,h}).
    \label{discrete_heat_orthogonality}
\end{equation}
On the left side of this equation, using the symmetry of the interface terms, we have
\begin{multline*}
     a^k(\textbf{u}^{k,h}, \textbf{e}^{k,h}) - a^k(\textbf{i}^k\textbf{R}^h\textbf{u}, \textbf{e}^{k,h}) = \int_{I} \Big\{ \big(\textbf{d}_t^k \textbf{e}^{k,h},
    \textbf{e}^{k,h} \big)_{\Omega} +  \big| \big| \big| \textbf{e}^{k,h} \big| \big| \big|_{\Omega} \\
     - \frac{1}{2} \big\langle  \nu_2 \partial_{\textbf{n}_2} (\textbf{e}_2^{k,h} - \textbf{I}_1^k \textbf{e}_2^{k,h}), \textbf{e}_1^{k,h} \big\rangle_{\Gamma} 
+ \frac{1}{2} \big\langle \nu_1 \partial_{\textbf{n}_1}( \textbf{I}_2^k \textbf{e}_1^{k,h} - \textbf{e}_1^{k,h}), \textbf{e}_2^{k,h} \big\rangle_{\Gamma}
\\ - \frac{1}{2} \big\langle \textbf{I}_1^k \textbf{e}_2^{k,h} - \textbf{e}_2^{k,h}, \nu_1 \partial_{\textbf{n}_1}\textbf{e}_1^{k,h} \big\rangle_{\Gamma}
+ \frac{1}{2} \big\langle \textbf{e}_1^{k,h} - \textbf{I}_2^k\textbf{e}_1^{k,h}, \nu_2 \partial_{\textbf{n}_2}\textbf{e}_2^{k,h} \big\rangle_{\Gamma}\\
- \frac{\gamma}{h} \big\langle \textbf{I}_1^k \textbf{e}_2^{k,h} - \textbf{e}_2^{k,h}, \textbf{e}_1^{k,h} \big\rangle_{\Gamma}
 + \frac{\gamma}{h} \big\langle \textbf{e}_1^{k,h} - \textbf{I}_2^k\textbf{e}_1^{k,h}, \textbf{e}_2^{k,h} \big\rangle_{\Gamma}\Big\} \diff t. 
\end{multline*}
Carrying out identical reasoning as in the previous proof leads us to 
\begin{equation*}
    \begin{aligned}
    & a^k(\textbf{u}^{k,h}, \textbf{e}^{k,h}) - a^k(\textbf{i}^k\textbf{R}^h\textbf{u}, \textbf{e}^{k,h}) = \int_{I} \Big\{ \big(\textbf{d}_t^k \textbf{e}^{k,h},
    \textbf{e}^{k,h} \big)_{\Omega} +  \big| \big| \big| \textbf{e}^{k,h} \big| \big| \big|_{\Omega}\Big\} \diff t. 
    \end{aligned}
\end{equation*}
To explain the disappearance of the interface terms, we refer to equation~(\ref{interface_zero}). We continue by analyzing the right side of the orthogonality identity~(\ref{discrete_heat_orthogonality}). Here we already omit the unnecessary interface terms including the projection operators
\begin{equation}
\begin{aligned}
  & a(\textbf{u}, \textbf{e}^{k,h}) - a^k(\textbf{i}^k\textbf{R}^h\textbf{u}, \textbf{e}^{k,h})\\
  & \qquad =  \int_{I} \Big\{\Big[\left(\textbf{d}_t \textbf{u} - \textbf{d}_t^k\textbf{i}^k\textbf{R}^h\textbf{u}, \textbf{e}^{k,h}\right)_{\Omega}\Big]^{(i)}
  + \Big[\boldsymbol{\nu}\left(\nabla (\textbf{u} - \textbf{i}^k\textbf{R}^h\textbf{u}), \nabla \textbf{e}^{k,h}\right)_{\Omega}\Big]^{(ii)} \\
  & \qquad  + \Big[ \frac{1}{2}\left\langle \nu_1 \partial_{\textbf{n}_1} (\textbf{u}_1 - \textbf{I}^k_1\textbf{R}^h_1\textbf{u}_1) - \nu_2 \partial_{\textbf{n}_2} (\textbf{u}_2 - \textbf{I}^k_2\textbf{R}^h_2\textbf{u}_2), \textbf{e}_2^{k,h} - \textbf{e}_1^{k,h}\right\rangle_{\Gamma}\Big]^{(iii)} \\
  & \qquad \Big[- \frac{1}{2} \left\langle (\textbf{u}_2 - \textbf{I}^k_2\textbf{R}^h_2\textbf{u}_2) - (\textbf{u}_1 - \textbf{I}^k_1\textbf{R}^h_1\textbf{u}_1), \nu_1 \partial_{\textbf{n}_1} \textbf{e}_1^{k,h} - \nu_2 \partial_{\textbf{n}_2} \textbf{e}_2^{k,h}\right\rangle_{\Gamma} \\
  & \qquad + \frac{\gamma}{h}\left\langle (\textbf{u}_2 - \textbf{I}^k_2\textbf{R}^h_2\textbf{u}_2) - (\textbf{u}_1 - \textbf{I}^k_1\textbf{R}^h_1\textbf{u}_1), \textbf{e}_2^{k,h} - \textbf{e}_1^{k,h}\right\rangle_{\Gamma}\Big]^{(iv)}  \Big\}\diff t.
\end{aligned}
\label{discrete_estimation_heat}
\end{equation}
Starting with the time discretization error \emph{(i)}, we can split it into contributions coming from the time and space 
\begin{equation*}
    \int_{I} \Big\{\big(\textbf{d}_t \textbf{u} - \textbf{d}_t^k \textbf{i}^k \textbf{R}^h \textbf{u}, \textbf{e}^{k,h} \big)_{\Omega} \Big\} \diff t 
       = \int_{I} \Big\{\big(\textbf{d}_t \textbf{u} - \textbf{d}_t^k \textbf{i}^k
    \textbf{u}, \textbf{e}^{k,h} \big)_{\Omega} + \big(\textbf{d}_t^k \textbf{i}^k \textbf{u} - \textbf{d}_t^k \textbf{i}^k \textbf{R}^h \textbf{u}, \textbf{e}^{k,h} \big)_{\Omega}\Big\} \diff t.
\end{equation*}
The first term has already been examined in the previous proof and based on~(\ref{time_equivalence}) it is equal to zero. Estimation of the second term directly follows from property~\ref{ritz_l2} in Corollary~\ref{ritz_properties} with the help of the Young's and Poincar\'e's inequalities.
\begin{multline}
  \Bigg |\int_{I^{n,m}_1} \big(\textbf{d}_t^k i^k_1 \textbf{u}_1 - \textbf{d}_t^k i^k_1 \textbf{R}^h_1 \textbf{u}_1, \textbf{e}^{k,h}_1 \big)_{\Omega_1} \diff t \Bigg | \\  
  \leq
  k^{n,m}_1 \max_{t \in I} \Big|\Big|\textbf{d}_t(\textbf{u}_1 - \textbf{R}_1^h \textbf{u}_1) \Big| \Big|_{\Omega_1} \Big| \Big| \textbf{e}_1^{k,h}(t^{n,m}_1) \Big| \Big|_{\Omega_1}  
  \leq c k^{n,m}_1  h^{2r + 2}  \max_{t \in I} \Big|\Big| \textbf{d}_t \nabla^{r + 1} \textbf{u}_1 \Big| \Big|^2_{\Omega_1} \\
  +\frac{1}{8} \int_{I^{n,m}_1} \nu_1^2\Big| \Big| \nabla \textbf{e}_1^{k,h} \Big| \Big|^2_{\Omega_1} \diff t.
  \label{time_derivative_space}
\end{multline}
We similarly split the Laplacian term~\emph{(ii)} in~(\ref{discrete_estimation_heat})
\begin{equation*}
    \int_{I} \boldsymbol{\nu}\left(\nabla (\textbf{u} - \textbf{i}^k\textbf{R}^h\textbf{u}), \nabla \textbf{e}^{k,h}\right)_{\Omega}  \diff t 
   = \int_{I} \Big\{\boldsymbol{\nu}\left(\nabla (\textbf{u} - \textbf{i}^k\textbf{u}), \nabla \textbf{e}^{k,h}\right)_{\Omega}  + \boldsymbol{\nu}\left(\nabla (\textbf{i}^k \textbf{u} - \textbf{i}^k\textbf{R}^h\textbf{u}), \nabla \textbf{e}^{k,h}\right)_{\Omega} \Big\} \diff t.
\end{equation*}
The first term was estimated in the previous proof by~(\ref{laplace_time}). The second term in the identity above can be also estimated by relying on property property~\ref{ritz_l2} in Corollary~\ref{ritz_properties} 
\begin{equation}
    \Bigg|\int_{I^{n,m}_1}\nu_1 \big(\nabla ( \textbf{i}_1^k\textbf{u}_1 - \textbf{i}^k_1 \textbf{R}^h_1  \textbf{u}_1), \nabla \textbf{e}_1^{k,h} \big)_{\Omega_1}\diff t \Bigg| 
    \leq c k^{n,m}_1 h^{2r} \Big|\Big| \nabla^{r + 1} \textbf{u}_1(t^{n,m}_1) \Big| \Big|^2_{\Omega_1} + 
    \frac{1}{8} \int_{I^{n,m}_1} \nu_1^2 \Big| \Big| \nabla \textbf{e}_1^{k,h} \Big| \Big|^2_{\Omega_1} \diff t. 
    \label{laplacian_space}
\end{equation}
We proceed to the coupling conditions in~(\ref{discrete_estimation_heat}). We estimate the first one~\emph{(iii)} on each macro time-step $I^n$ 
\begin{equation*}
    \begin{aligned}
     & \Bigg|\int_{I^n} \frac{1}{2}\left\langle \nu_1 \partial_{\textbf{n}_1} (\textbf{u}_1 - \textbf{I}^k_1\textbf{R}^h_1\textbf{u}_1) - \nu_2 \partial_{\textbf{n}_2} (\textbf{u}_2 - \textbf{I}^k_2\textbf{R}^h_2\textbf{u}_2), \textbf{e}_2^{k,h} - \textbf{e}_1^{k,h}\right\rangle_{\Gamma} \diff t \Bigg| \\ & \quad \leq \int_{I^n} \Bigg\{\frac{h}{2\gamma} \Big|\Big|\nu_1 \partial_{\textbf{n}_1} (\textbf{u}_1 - \textbf{I}^k_1\textbf{R}^h_1\textbf{u}_1) - \nu_2 \partial_{\textbf{n}_2} (\textbf{u}_2 - \textbf{I}^k_2\textbf{R}^h_2\textbf{u}_2)\Big|\Big|_{\Gamma}^2 + \frac{\gamma}{8h} \left|\left|\textbf{e}_2^{k,h} - \textbf{e}_1^{k,h}\right|\right|_{\Gamma}^2 \Bigg\}\diff t.
    \end{aligned}
\end{equation*}
The normal derivatives are  estimated by splitting the errors similarly
\begin{multline}
  \int_{I^n}\Big|\Big|\nu_1 \partial_{\textbf{n}_1} (\textbf{u}_1 - I^k_1\textbf{R}^h_1\textbf{u}_1) - \nu_2 \partial_{\textbf{n}_2} (\textbf{u}_2 - I^k_2\textbf{R}^h_2\textbf{u}_2)\Big|\Big|_{\Gamma}^2 \diff t \\
  \leq \sum_{j=1}^2 \sum_{m = 1}^{N^n_j} c_jk^{n,m}_j(\nu_j)^2\Bigg\{(k^{n,m}_j)^2\max_{t \in I} \left|\left|\textbf{d}_t \partial_{\textbf{n}_1}\textbf{u}_j\right|\right|_{\Gamma}^2 
  +  h^{2r- 1}\left|\left| \nabla^{r+1}\textbf{u}_j(t^{n,m}_j)\right|\right|_{\Omega_j}^2\Bigg\}
  \label{heat_interface_estimation}
\end{multline}
Estimation of the space component follows from property~\ref{ritz_interface} in Corollary~\ref{ritz_properties}. The time component was estimated in~(\ref{interface_time}). The analysis of the remaining two interface terms~\emph{(iv)} is similar. Since the last one is slightly simpler, we will take it as an example. We have, given the continuity of $\textbf{u}$ and using an identical set of arguments as in the previous proof
\begin{multline*}
  \int_{I^n} \frac{\gamma}{h}\left\langle (\textbf{u}_2 - \textbf{I}^k_2\textbf{R}^h_2\textbf{u}_2) - (\textbf{u}_1 - \textbf{I}^k_1\textbf{R}^h_1\textbf{u}_1), \textbf{e}_2^{k,h} - \textbf{e}_1^{k,h}\right\rangle_{\Gamma} \diff t \\
  = \int_{I^n} \frac{\gamma}{h}\left\langle  \textbf{I}^k_1\textbf{R}^h_1\textbf{u}_1 - \textbf{I}^k_2\textbf{R}^h_2\textbf{u}_2, \textbf{e}_2^{k,h} - \textbf{e}_1^{k,h}\right\rangle_{\Gamma} \diff t\\
  = \int_{I^n} \frac{\gamma}{h}\left\langle  \textbf{R}^h_1\textbf{u}_1 - \textbf{R}^h_2\textbf{u}_2, \textbf{e}_2^{k,h} - \textbf{e}_1^{k,h}\right\rangle_{\Gamma} \diff t.
\end{multline*}
Finally, given the continuity of the Ritz operator (\ref{ritz_continuity}), it holds
\begin{equation}
\begin{aligned}
    &\int_{I^n} \frac{\gamma}{h}\left\langle  \textbf{R}^h_1\textbf{u}_1 - \textbf{R}^h_2\textbf{u}_2, \textbf{e}_2^{k,h} - \textbf{e}_1^{k,h}\right\rangle_{\Gamma} \diff t = \textbf{0}.
\end{aligned}
\label{heat_interface_estimation_zero}
\end{equation}
The last interface term in~(\ref{discrete_estimation_heat}) can be estimated similarly. That ends the proof. 
\end{proof}
The theorem is equivalent to 
\begin{multline*}
  \big| \big| \textbf{e}^{k,h}(t^N)\big|\big| + \int_I \big|\big| \big| \textbf{e}^k\big|\big|\big| \diff t\\
  =
    \sum_{j=1}^2 \Big\{ \mathcal{O}\big(k_j\left| \left| \textbf{d}_t \nabla \textbf{u}_j \right| \right|_{\Omega_j} \big) + \mathcal{O}\big(k_j h^{\frac{1}{2}} \left|\left|\textbf{d}_t \partial_{\textbf{n}_j}\textbf{u}_j\right|\right|_{\Gamma} \big) \\
     + \mathcal{O}\big(h^{r + 1}  \big|\big| \textbf{d}_t \nabla^{r + 1} \textbf{u}_j \big| \big|_{\Omega_j}\big)  +
    \mathcal{O}\big(h^r \left| \left|\nabla^{r + 1} \textbf{u}_j \right| \right|_{\Omega_j} \big)\Big\}.
\end{multline*}
Again, we were able to fully decouple this system. In fact, we were able to equip the terms $\left|\left|\textbf{d}_t \partial_{\vec{\textbf{n}}_1}\textbf{u}_1\right|\right|_{\Gamma} $ and $\left|\left|\textbf{d}_t \partial_{\vec{\textbf{n}}_2}\textbf{u}_2\right|\right|_{\Gamma} $ with an additional half an order of convergence in space compared to the semi-discrete case. That being said, a comprehensive comparison between the fully discrete and semi-discrete cases is not possible since both of the inequalities are proved in different norms. Overall, we preserved the linear convergence in time.



\section{Coupling of Stokes Equations}

 As our third and final problem, we consider a system composed of two time-dependent Stokes equations. Each of them has a separate kinematic viscosity $(\nu_1, \nu_2)^T = \boldsymbol{\nu}$. Velocity $\textbf{u} = (\textbf{u}_1, \textbf{u}_2)^T: \Omega \times I  \to \mathbb{R}^d $, $\textbf{u}_j \in X(H^1(\Omega_j))^d$ and pressure $\textbf{p} = (p_1, p_2)^T: \Omega \times I  \to \mathbb{R}$, $p_j \in X(L^2(\Omega_j))$ for $j = 1,2$ are solutions to the system
\begin{equation}
  \textnormal{div} \ \textbf{u}_j = 0,\quad
  \partial_t \textbf{u}_j - 2 \nu_j \textnormal{div}\,\dot{\boldsymbol{\epsilon}}(\textbf{u}_j)  + \nabla p_j= \textbf{f}_j\text{ in }\Omega_j,\quad j=1,2
\end{equation}
where 
\begin{equation*}
  \dot{\boldsymbol{\epsilon}}(\textbf{u}) = \frac{1}{2} \left(\nabla \textbf{u} + \nabla \textbf{u}^T \right).    
\end{equation*}
On the outer boundary we set $\textbf{u}_1 = \textbf{u}_2 = \textbf{0}$. Also at the initial time, we impose $\textbf{u}_1(0) = \textbf{u}_2(0) = \textbf{0}$. On the interface we set
\begin{equation}
  \textbf{u}_1 = \textbf{u}_2\text{ and }
  \sigma_1(\textbf{u}_1, p_1) \cdot \textbf{n}_1= -  \sigma_2(\textbf{u}_2, p_2)\cdot \textbf{n}_2  \hspace*{0.2  cm}\text{ on }\Gamma,
\end{equation}
where the stress tensors  $\boldsymbol{\sigma} = (\sigma_1, \sigma_2)^T$ are given by
\begin{equation*}
  \sigma_j(\textbf{u}_j, p_j) = 2 \nu_j \dot{\boldsymbol{\epsilon}}(\textbf{u}_j) - p_j I,\quad j=1,2.
\end{equation*}
As test functions, we take $\boldsymbol{\varphi} = (\boldsymbol{\varphi}_1, \boldsymbol{\varphi}_2)^T$, $\boldsymbol{\varphi}_j \in X(H^1(\Omega_j))^d$ and $\boldsymbol{\psi} = (\psi_1, \psi_2)^T$, $\psi_j \in X(L^2(\Omega_j))$ for $j = 1,2$. We define the incompressibility form as
\begin{equation}
  \begin{aligned}
    & b(\textbf{u}, \boldsymbol{\psi}) = \int_I \Big\{- \left(\textnormal{div} \ \textbf{u}, \boldsymbol{\psi} \right)_{\Omega} +\frac{1}{2} \left\langle  \psi_2 \textbf{n}_2 -  \psi_1 \textbf{n}_1, \textbf{u}_2 - \textbf{u}_1  \right\rangle_{\Gamma} \Big\} \diff t.
  \end{aligned}
\end{equation}
In this formulation, we again follow the work presented in \cite{DG}. We note that the additional interface term does not violate the incompressibility condition since the exact solution is continuous across the interface and therefore 
\begin{equation*}
  \left\langle  \psi_2 \textbf{n}_2 -  \psi_1 \textbf{n}_1, \textbf{u}_2 - \textbf{u}_1  \right\rangle_{\Gamma} = 0. 
\end{equation*}
Given that, we present the variational problem
\begin{equation}
  a(\textbf{u}, \boldsymbol{\varphi}) + b(\boldsymbol{\varphi}, \textbf{p}) - b(\textbf{u}, \boldsymbol{\psi}) = \int_{I} \big(\textbf{f}, \boldsymbol{\varphi} \big)_{\Omega} \diff t,
  \label{continuous_stokes}
\end{equation}
where the form $a(\cdot, \cdot)$ has a similar definition as in the case of the heat equation
\begin{equation*}
  \begin{aligned}
    a(\textbf{u}, \boldsymbol{\varphi})
    & \coloneqq  \int_{I} \Big\{\left(\textbf{d}_t \textbf{u}, \boldsymbol{\varphi}\right)_{\Omega} + 2\boldsymbol{\nu}\left( \dot{\boldsymbol{\epsilon}}(\textbf{u}), \nabla \boldsymbol{\varphi}\right)_{\Omega}
    + \gamma\left\langle \textbf{u}_2 - \textbf{u}_1, \boldsymbol{\varphi}_2 - \boldsymbol{\varphi}_1\right\rangle_{\Gamma}\\
    & \qquad  +  \left\langle \nu_1   \dot{\epsilon}(\textbf{u}_1) \cdot \textbf{n}_1   - \nu_2 \dot{\epsilon}(\textbf{u}_2) \cdot \textbf{n}_2, \boldsymbol{\varphi}_2 - \boldsymbol{\varphi}_1\right\rangle_{\Gamma} 
     - \left\langle \textbf{u}_2 - \textbf{u}_1, \nu_1 \dot{\epsilon}(\boldsymbol{\varphi}_1) \cdot \textbf{n}_1  - \nu_2  \dot{\epsilon}(\boldsymbol{\varphi}_2) \cdot \textbf{n}_2\right\rangle_{\Gamma}   \Big\}\diff t
  \end{aligned}
\end{equation*}
\subsection{Discretization in Time}
 Similarly, we can construct a discrete incompressibility form. We have to pay attention to the position of the trial and test functions. Therefore we define two versions of this form
\begin{equation}
  \begin{aligned}
    \bar{b}^k(\textbf{u}^k, \boldsymbol{\psi}^k) &\coloneqq \int_I \Big\{- \left(\textnormal{div} \ \textbf{u}^k, \boldsymbol{\psi}^k \right)_{\Omega} -\frac{1}{2} \left\langle \textbf{n}_1 \psi_1^k, \textbf{I}_1^k \textbf{u}_2^k - \textbf{u}_1^k  \right\rangle_{\Gamma}  +\frac{1}{2} \left\langle \textbf{n}_2 \psi_2^k, \textbf{u}_2^k - \textbf{I}_2^k\textbf{u}_1^k  \right\rangle_{\Gamma} \Big\}\diff t, \\
    b^k(\boldsymbol{\varphi}^k, \textbf{p}^k) &\coloneqq \int_I \Big\{- \left(\textnormal{div} \ \boldsymbol{\varphi}^k, \textbf{p}^k \right)_{\Omega} -\frac{1}{2} \left\langle \textbf{n}_2 I_1^k p_2^k - \textbf{n}_1 p_1^k, \boldsymbol{\varphi}_1^k  \right\rangle_{\Gamma} 
    + \frac{1}{2} \left\langle \textbf{n}_2 p_2^k - \textbf{n}_1 I_2^k p_1, \boldsymbol{\varphi}_2^k  \right\rangle_{\Gamma} \Big\}\diff t.
  \end{aligned}
\end{equation}
Together with the form 
\begin{equation*}
  \begin{aligned}
    a^{k}(\textbf{u}^{k}, \boldsymbol{\varphi}^{k}) & \coloneqq  \int_{I} \Big\{\big(\textbf{d}_t^k \textbf{u}^{k}, \boldsymbol{\varphi}^{k}\big)_{\Omega} + 2 \boldsymbol{\nu}\left(\dot{\boldsymbol{\epsilon}}( \textbf{u}^{k}), \nabla \boldsymbol{\varphi}^{k}\right)_{\Omega} \\ & \qquad 
    - \left\langle \nu_1  \dot{\epsilon} (\textbf{u}_1^k) \cdot \textbf{n}_1 - \nu_2  \dot{\epsilon}( \textbf{I}_1^k \textbf{u}_2^k) \cdot \textbf{n}_2, \boldsymbol{\varphi}_1^k \right\rangle_{\Gamma} 
    + \left\langle \nu_1  \dot{\epsilon} (\textbf{I}_2^k \textbf{u}_1^k) \cdot \textbf{n}_1 - \nu_2 \dot{\epsilon}(\textbf{u}_2^k) \cdot \textbf{n}_2, \boldsymbol{\varphi}_2^k \right\rangle_{\Gamma}
    \\ & \qquad - \left\langle \textbf{I}_1^k \textbf{u}_2^k - \textbf{u}_1^k, \nu_1  \dot{\epsilon}(\boldsymbol{\varphi}_1^k) \cdot \textbf{n}_1 \right\rangle_{\Gamma}
    +\left\langle \textbf{u}_2^k - \textbf{I}_2^k\textbf{u}_1^k, \nu_2  \dot{\epsilon}(\boldsymbol{\varphi}_2^k) \cdot \textbf{n}_2 \right\rangle_{\Gamma}\\
    &\qquad - \gamma \left\langle \textbf{I}_1^k \textbf{u}_2^k - \textbf{u}_1^k, \boldsymbol{\varphi}_1^k \right\rangle_{\Gamma}
    + \gamma \left\langle \textbf{u}_2^k -\textbf{I}_2^k\textbf{u}_1^k, \boldsymbol{\varphi}_2^k \right\rangle_{\Gamma}
    \Big\}\diff t,
  \end{aligned}
\end{equation*}
this defines the semi-discrete variational problem
\begin{equation}
  a^k(\textbf{u}^k, \boldsymbol{\varphi}^k) + b^k(\boldsymbol{\varphi}^k, \textbf{p}^k) - \bar{b}^k(\textbf{u}^k, \boldsymbol{\psi}^k) = \int_{I} \big(\textbf{f}, \boldsymbol{\varphi}^k \big)_{\Omega} \diff t.
  \label{semi-discrete_stokes}
\end{equation}
From this formulation, we can derive semi-discrete coupling conditions
\begin{equation}
  \begin{aligned}
    0 &=  \int_{I_n} \big\langle \textbf{u}_2^k - \textbf{u}_1^k, \boldsymbol{\varphi}_1^k \big\rangle_{\Gamma} \diff t = \int_{I_n} \big\langle \textbf{u}_2^k - \textbf{u}_1^k, \boldsymbol{\varphi}_2^k \big\rangle_{\Gamma} \diff t, \\
    0 &= \int_{I_n} \big\langle  \sigma_1(\textbf{u}_1, p_1) \cdot \textbf{n}_1 +  \sigma_2(\textbf{u}_2, p_2) \cdot \textbf{n}_2, \boldsymbol{\varphi}_j^k \rangle_{\Gamma} \diff t,\quad j=1,2,
  \end{aligned}
  \label{weak_coupling}
\end{equation}
for any $\boldsymbol{\varphi}_j^k \in (X^{k}_j)^d$. As it turns out, the theorem from the previous sections can be easily extended to the Stokes equation as well

\begin{theorem}
  Let $\textbf{u} \in X(H^1_0(\Omega))^d$, $\textbf{u}_j \in W^{1, \infty}(H^2(\Omega_j))^d$, $p_j \in W^{1, \infty}(L^2(\Omega_j))$ for $j=1,2$ be continuous solutions to~(\ref{continuous_stokes}) and $\textbf{u}^k \times \textbf{p}^k \in \big(X^k\big)^{d + 1}$ their semi-discrete counterparts and solutions to~(\ref{semi-discrete_stokes}), then the following estimate holds
  \begin{multline*}
    \big| \big| \textbf{e}^k(t^N)\big|\big|^2_{\Omega} + \int_I \boldsymbol{\nu}^2\left|\left|\nabla \textbf{e}^k \right|\right|^2_{\Omega} \diff t
    \\
    \leq C\sum_{j=1}^2 \sum_{n = 1}^{N}\sum_{m = 1}^{N^n_j}  \Bigg\{  (k^{n, m}_j)^3 \max_{t \in I}
    \left| \left| \textbf{d}_t \dot{\epsilon} (\textbf{u}_j) \right| \right|^2_{\Omega_j} + (k^{n,m}_j)^3 \max_{t \in I} \left| \left|d_t p_j\right| \right|_{\Omega_j}^2 \\
    + (k^{n, m}_j)^3 \max_{t \in I} \left|\left| \textbf{d}_t \sigma_j(\textbf{u}_j, p_j) \cdot \textbf{n}_j\right|\right|_{\Gamma}^2 \Bigg\},
  \end{multline*}
  where the errors $\textbf{e}^k = (\textbf{e}_1^k, \textbf{e}_2^k)^T$ are defined as
  $\textbf{e}_j^k \coloneqq \textbf{u}_j^k - \textbf{i}_j^k \textbf{u}_j$ for $j=1,2$.
  \label{theorem_semidiscrete_stokes}
\end{theorem}

\begin{proof} 
  By using Galerkin orthogonality
  \begin{equation}
    \begin{aligned}
      & a^k(\textbf{u}^k, \textbf{e}^k) + b^k(\textbf{e}^k, \textbf{p}^k) - \int_{I} \Big\{\big(\textbf{d}_t^k \textbf{i}^k \textbf{u}, \textbf{e}^k\big)_{\Omega} + 2\boldsymbol{\nu} \big(\dot{\boldsymbol{\epsilon}}( \textbf{i}^k \textbf{u}), \nabla{\textbf{e}^k} \big)_{\Omega} \Big\} \diff t \\ & \qquad= a(\textbf{u}, \textbf{e}^k) + b(\textbf{e}^k, \textbf{p})- \int_{I} \Big\{\big(\textbf{d}_t^k \textbf{i}^k \textbf{u}, \textbf{e}^k\big)_{\Omega} + 2\boldsymbol{\nu} \big(\dot{\boldsymbol{\epsilon}}( \textbf{i}^k \textbf{u}), \nabla{\textbf{e}^k} \big)_{\Omega} \Big\} \diff t.
    \end{aligned}
    \label{semi-discrete-stokes-orthogonality}
  \end{equation}
  For any $\boldsymbol{\psi}^k \in X^k$, we have
  \begin{equation*}
    \int_I \left(\textnormal{div} \ \textbf{u}, \boldsymbol{\psi}^k \right)_{\Omega} \diff t = \int_I \left(\textnormal{div} \ \textbf{u}^k, \boldsymbol{\psi}^k \right)_{\Omega} \diff t = 0.
  \end{equation*}
  Knowing that the semi-discrete pressure is piecewise constant in time, we can claim that
  \begin{multline}
      \int_I\left(\textnormal{div} \ \textbf{e}^k, \textbf{p}^k\right)_{\Omega}\diff t  = \int_I \Big\{\left(\textnormal{div} \ \textbf{u}^k, \textbf{p}^k\right)_{\Omega} -  \left(\textnormal{div} \ \textbf{i}^k \textbf{u}, \textbf{p}^k\right)_{\Omega} \Big\} \diff t \\
       = \int_I \Big\{\left(\textnormal{div} \ \textbf{u}^k, \textbf{p}^k\right)_{\Omega} - \textbf{i}^k \left(\textnormal{div} \ \textbf{u}, \textbf{p}^k\right)_{\Omega} \Big\} \diff t = 0.
    \label{interpolation_div_zero}
  \end{multline}
  Therefore, on the left side of (\ref{semi-discrete-stokes-orthogonality}), we have
  \begin{multline*}
    a^k(\textbf{u}^k, \textbf{e}^k) + b^k(\textbf{e}^k, \textbf{p}^k) - \int_{I} \Big\{\big(\textbf{d}_t^k \textbf{i}^k \textbf{u}, \textbf{e}^k\big)_{\Omega} + 2\boldsymbol{\nu} \big(\dot{\boldsymbol{\epsilon}}( \textbf{i}^k \textbf{u}), \nabla{\textbf{e}^k} \big)_{\Omega} \Big\} \diff t \\
    =  \int_I \big(\textbf{d}_t^k \textbf{e}^{k,h}, \textbf{e}^{k,h} \big)_{\Omega} \diff t + \int_I 2\boldsymbol{\nu} \big(\dot{\boldsymbol{\epsilon}}( \textbf{e}^{k,h}), \nabla\textbf{e}^{k,h} \big)_{\Omega} \diff t.
  \end{multline*}
  We can similarly show that 
  \begin{equation*}
    \int_I\left(\textnormal{div} \ \textbf{e}^k, \textbf{p}\right)_{\Omega}\diff t = \int_I\left(\textnormal{div} \ \textbf{e}^k, \textbf{p} - \textbf{i}^k\textbf{p}\right)_{\Omega}\diff t. 
  \end{equation*}
  Indeed, it holds
  \begin{multline*}
    \int_I\left(\textnormal{div} \ \textbf{e}^k, \textbf{i}^k \textbf{p}\right)_{\Omega}\diff t  =  \int_I \Big\{\left(\textnormal{div} \ \textbf{u}^k, \textbf{i}^k \textbf{p}^k\right)_{\Omega} -  \left(\textnormal{div} \ \textbf{i}^k \textbf{u}, \textbf{i}^k \textbf{p}^k\right)_{\Omega} \Big\} \diff t\\
    =  -\int_I \textbf{i}^k\left(\textnormal{div} \  \textbf{u}, \textbf{i}^k \textbf{p}^k\right)_{\Omega} \diff t = 0. 
  \end{multline*}
  Further, on each interval $I^{n,m}_1$ we have
  \begin{equation*}
    \begin{aligned}
      \Bigg|\int_{I^{n,m}_1}\left(\textnormal{div} \ \textbf{e}^k_1, p_1 - i_1^k p_1\right)_{\Omega_1}\diff t \Bigg|  \leq c (k^{n,m}_1)^3\max_{t \in I} \left| \left| d_t p_1\right| \right|_{\Omega_1}^2 + \frac{1}{8} \nu_1 \int_{I^{n,m}_1} \left| \left|\nabla \textbf{e}_1^k \right| \right|_{\Omega_1}^2 \diff t.
    \end{aligned}
  \end{equation*}
  We will symmetrize the Laplacian term in~(\ref{semi-discrete-stokes-orthogonality}) and use the first Korn inequality
  \begin{equation*}
    \int_I 2\boldsymbol{\nu} \big(\dot{\boldsymbol{\epsilon}}( \textbf{e}^{k,h}), \nabla\textbf{e}^{k,h} \big)_{\Omega} \diff t 
    = \int_I 2\boldsymbol{\nu} \big| \big| \dot{\boldsymbol{\epsilon}}( \textbf{e}^{k,h}) \big|\big|_{\Omega}^2 \diff t  \geq c_K \int_I \boldsymbol{\nu} \big| \big| \nabla \textbf{e}^{k,h} \big|\big|_{\Omega}^2  \diff t .
  \end{equation*}
  By $c_K$ we denote the constant from Korn's inequality.
  The rest directly follows from Theorem~\ref{theorem_semidiscrete_heat} simply by using the appropriate Neumann coupling conditions on the interface (\ref{interface_time}) and replacing $\nabla \textbf{e}^{k,h}$ with $\dot{\boldsymbol{\epsilon}}(\textbf{e}^{k,h})$ in~(\ref{laplace_time}).
\end{proof}
We just showed that
\begin{multline*}
  \big| \big| \textbf{e}^k(t^N)\big|\big| + \int_I\boldsymbol{\nu}\big| \big| \textbf{e}^k\big|\big| \diff t \\
   = \sum_{j=1}^2 \mathcal{O}\big(k_j\left| \left| \textbf{d}_t \dot{\epsilon}(\textbf{u}_j) \right| \right|_{\Omega_j} \big) + \mathcal{O}\big(k_j \left| \left|d_t p_j\right| \right|_{\Omega_j} \big)  + \mathcal{O}\big(k_j \left|\left| \textbf{d}_t \sigma_j(\textbf{u}_j, p_j) \cdot \textbf{n}_j \right|\right|_{\Gamma} \big) 
\end{multline*}
These results are analogous to what we were able to show in Theorem~\ref{theorem_semidiscrete_heat}. The differences include replacing $\nabla \textbf{e}^{k,h}$ with $\dot{\boldsymbol{\epsilon}}(\textbf{e}^{k,h})$ and adding new pressure terms. The newly introduced volume pressure terms are also decoupled. We managed to preserve the optimal linear convergence rate. 


\subsection{Discretization in Time and Space}

 We can establish a similar estimate for the fully discrete coupled Stokes equations. We consider classical inf-sup stable Taylor-Hood elements, where $\textbf{u}^{k,h} \times \textbf{p}^{k,h}\in \big(X^{k,h}(r)\big)^d \times X^{k,h}(r-1)$ for $ r \geq 2$. The fully discrete variational formulation reads as 
\begin{equation}
  a^{k,h}(\textbf{u}^{k,h}, \boldsymbol{\varphi}^{k,h}) + b^k(\boldsymbol{\varphi}^{k,h}, \textbf{p}^{k,h}) - \bar{b}^k(\textbf{u}^{k,h}, \boldsymbol{\psi}^{k,h}) = \int_{I} \big(\textbf{f}, \boldsymbol{\varphi}^{k,h} \big)_{\Omega} \diff t,
  \label{fully_discrete_stokes}
\end{equation}
where
\begin{equation*}
  \begin{aligned}
    & a^{k, h}(\textbf{u}^{k,h}, \boldsymbol{\varphi}^{k,h}) \coloneqq \int_{I} \Big\{\big(\textbf{d}_t^k \textbf{u}^{k,h}, \boldsymbol{\varphi}^{k,h}\big)_{\Omega} + 2\boldsymbol{\nu}\left(\dot{\boldsymbol{\epsilon}} (\textbf{u}^{k,h}), \nabla \boldsymbol{\varphi}^{k,h}\right)_{\Omega} \\ & \qquad 
    - \left\langle \nu_1 \dot{\epsilon} (\textbf{u}^{k,h}_1) \cdot \textbf{n}_1 - \nu_2  \dot{\epsilon}( \textbf{I}_1^k \textbf{u}_2^{k,h}) \cdot \textbf{n}_2, \boldsymbol{\varphi}_1^{k,h} \right\rangle_{\Gamma} 
    + \left\langle \nu_1 \dot{\epsilon}( \textbf{I}_2^k \textbf{u}_1^{k,h}) \cdot \textbf{n}_1 - \nu_2  \dot{\epsilon} (\textbf{u}_2^{k,h}) \cdot \textbf{n}_2, \boldsymbol{\varphi}_2^{k,h} \right\rangle_{\Gamma}
    \\ & \qquad - \left\langle \textbf{I}_1^k \textbf{u}_2^{k,h} - \textbf{u}_1^{k,h}, \nu_1 \dot{\epsilon} (\boldsymbol{\varphi}_1^{k,h}) \cdot \textbf{n}_1\right\rangle_{\Gamma}
    + \left\langle \textbf{u}_2^{k,h} - \textbf{I}_2^k\textbf{u}_1^{k,h}, \nu_2  \dot{\epsilon}(\boldsymbol{\varphi}_2^{k,h}) \cdot \textbf{n}_2 \right\rangle_{\Gamma}
    \\ & \qquad - \frac{\gamma}{h} \left\langle \textbf{I}_1^k \textbf{u}_2^{k,h} - \textbf{u}_1^{k,h}, \boldsymbol{\varphi}_1^{k,h} \right\rangle_{\Gamma}
    + \frac{\gamma}{h} \left\langle \textbf{u}_2^{k,h} -\textbf{I}_2^k\textbf{u}_1^{k,h}, \boldsymbol{\varphi}_2^{k,h} \right\rangle_{\Gamma}
    \Big\}\diff t.
  \end{aligned}
\end{equation*}
For the Stokes problem, we will use a modified version of the Ritz projection operator~(\ref{ritz}) to account for the incompressibility condition
\begin{equation*}
  \begin{aligned}
    \big(\nabla \textbf{R}^h \textbf{u}, \nabla \boldsymbol{\varphi}^{k,h}\big)_{\Omega} - \big(\textbf{q},  \textnormal{div} \ \boldsymbol{\varphi}^{k,h}\big)_{\Omega} &= \big(\nabla \textbf{u}, \nabla \boldsymbol{\varphi}^{k,h}\big)_{\Omega}, && \boldsymbol{\varphi}^{k,h} \in (Y^{k,h}(r))^d\\
    \big(\textnormal{div} \ \textbf{R}^h \textbf{u} , \boldsymbol{\psi}^{k,h }\big)_{\Omega} &= 0, && \boldsymbol{\psi}^{k,h} \in Y^{k,h}(r-1). 
  \end{aligned}
\end{equation*}
All of the properties established in Corollary~\ref{ritz_properties} still hold. The newly introduced pressure term $\textbf{q}$ is only a Lagrange multiplier needed to project the solution $\textbf{u}$ into the space of divergence-free functions. We will not come back to it in the proofs. However, for the pressure $\textbf{p}$, we will use an additional projection operator $\textbf{I}^h = (I_1^h, I_2^h)^T$ given by
\begin{equation}
  \big(\textbf{I}^h \textbf{p}, \boldsymbol{\psi}^{k,h}\big)_{\Omega} = \big(\textbf{p}, \boldsymbol{\psi}^{k,h}\big)_{\Omega} \hspace*{0.3 cm}
  \textnormal{ for all } \hspace*{0.3 cm} \boldsymbol{\psi}^{k,h} \in X^{k,h}(r-1).
  \label{pressure_projection}
\end{equation}
Below we list some of the useful properties.
\begin{corollary} Given $p_j \in L(\bar{I}, H^{r}(\Omega_j))$ for $j=1,2$, the projection operator given by~(\ref{pressure_projection}) has the following properties:
  \begin{enumerate}[label=(\roman*)]
  \item 
    $
    \big| \big| \textbf{p} - \textbf{I}^h \textbf{p}\big|\big|_{\Omega} \leq c_1 h^{r} \big|\big| \nabla^{r} \textbf{p}_1\big|\big|_{\Omega_1} + c_2 h^{r} \big|\big| \nabla^{r} \textbf{p}_2\big|\big|_{\Omega_2},
    $
    \label{pressure_l2}
    
  \item 
    $
    \big| \big|( p_1 - I_1^h p_1) \textbf{n}_1\big|\big|_{\Gamma} +  \big| \big|( p_2 - I_2^h p_2) \textbf{n}_2\big|\big|_{\Gamma}
     \leq c_1 h^{r - \frac{1}{2}} \big|\big| \nabla^{r} p_1\big|\big|_{\Omega_1} + c_2 h^{r - \frac{1}{2}} \big|\big| \nabla^{r} p_2\big|\big|_{\Omega_2}$. \label{pressure_interface}
  \end{enumerate}
  \label{pressure_properties} 
\end{corollary}
We proceed to the velocity error estimation for the fully discrete problem. 
\begin{theorem}[A priori velocity estimate for the fully discrete Stokes problem]\label{thm:stokes:vel}
  Let $\textbf{u} \in X(H^1_0(\Omega))^d$, $\textbf{u}_j \in W^{1, \infty}(H^{r+1}(\Omega_j)^d$, $p_j \in W^{1, \infty}(H^{r}(\Omega_j)$ for $j=1,2$ be continuous solutions to~(\ref{continuous_stokes}) and $\textbf{u}^{k,h} \times \textbf{p}^{k,h}\in \big(X^{k,h}(r)\big)^d \times X^{k,h}(r-1)$ their discrete counterparts and solutions to~(\ref{fully_discrete_stokes}), then the following estimate holds
  \begin{multline*}
    \big| \big| \textbf{e}^{k,h}(t^N)\big|\big|^2_{\Omega}  + \int_I  \big|\big|\big|\textbf{e}^{k,h} \big|\big|\big|^2_{\Omega} \diff t  \\
    \leq C\sum_{j=1}^2
    \sum_{n = 1}^{N}\sum_{m = 1}^{N^n_j}  \Bigg\{(k^{n, m}_j)^3 \max_{t \in I} \left| \left| \textbf{d}_t \dot{\epsilon} (\textbf{u}_j) \right| \right|^2_{\Omega_j} + (k^{n,m}_j)^3 \max_{t \in I} \left| \left|d_t p_j\right| \right|_{\Omega_j}^2  \\
    + (k^{n, m}_j)^3h \max_{t \in I} \left|\left| \textbf{d}_t \sigma_j(\textbf{u}_j, p_j) \cdot \textbf{n}_j \right|\right|_{\Gamma}^2 \\
    + k^{n,m}_j  h^{2r + 2}  \max_{t \in I} \Big|\Big| \textbf{d}_t \nabla^{r + 1} \textbf{u}_j \Big| \Big|^2_{\Omega_j}
    +k_j^{n,m} h^{2r} \left| \left|\nabla^{r + 1} \textbf{u}_j(t^{n,m}_j) \right| \right|^2_{\Omega_j} 
    + k_j^{n,m} h^{2r} \left| \left|\nabla^{r} p_j(t^{n,m}_1) \right| \right|^2_{\Omega_j}
    \Bigg\} 
  \end{multline*}
  where the errors $\textbf{e}^k = (\textbf{e}_1^k, \textbf{e}_2^k)^T$, $\boldsymbol{\eta}^k = (\eta_1^k, \eta_2^k)^T$ are defined as
  $\textbf{e}_j^{k,h} \coloneqq \textbf{u}_j^{k,h} - \textbf{i}_j^k \textbf{R}_j^h \textbf{u}_j$ and
  $\eta_j^{k,h} \coloneqq  p_j^{k,h} - i_j^k I_j^h p_j$ for $j=1,2$.
  \label{theorem_discrete_stokes}
\end{theorem}
\begin{proof}
  The Galerkin orthogonality gives us 
  \begin{equation}
    \begin{aligned}
      &a^k(\textbf{e}^{k,h}, \textbf{e}^{k,h}) + b^k(\textbf{e}^{k,h}, \boldsymbol{\eta}^{k,h}) - \bar{b}^k(\textbf{e}^{k,h}, \boldsymbol{\eta}^{k,h}) \\ & \qquad = a(\textbf{u}, \textbf{e}^{k,h}) - a^k(\textbf{i}^k\textbf{R}^h\textbf{u}, \textbf{e}^{k,h}) + b(\textbf{e}^{k,h}, \textbf{p}) - b^k(\textbf{e}^{k,h}, \textbf{i}^k \textbf{I}^h \textbf{p}) \\ & \qquad- b(\textbf{u}, \boldsymbol{\eta}^{k,h}) + \bar{b}^k(\textbf{i}^k \textbf{R}^h \textbf{u}, \boldsymbol{\eta}^{k,h}).
    \end{aligned}
    \label{fully-discrete-stokes-orthogonal}
  \end{equation}
  We can show that the left side is equal to 
  $$a^k(\textbf{e}^{k,h}, \textbf{e}^{k,h}) + b^k(\textbf{e}^{k,h}, \boldsymbol{\eta}^{k,h}) - \bar{b}^k(\textbf{e}^{k,h}, \boldsymbol{\eta}^{k,h})  = a^k(\textbf{e}^{k,h}, \textbf{e}^{k,h}).$$
  On the right side of~(\ref{fully-discrete-stokes-orthogonal}), we have
  \begin{equation*}
    \begin{aligned}
      & a^k(\textbf{e}^{k,h}, \textbf{e}^{k,h}) =  \int_{I} \Big\{\left(\textbf{d}_t \textbf{u} - \textbf{d}_t^k\textbf{i}^k\textbf{R}^h\textbf{u}, \textbf{e}^{k,h}\right)_{\Omega}  + 2\boldsymbol{\nu}\left(\dot{\boldsymbol{\epsilon}} (\textbf{u} - \textbf{i}^k\textbf{R}^h\textbf{u}), \nabla \textbf{e}^{k,h}\right)_{\Omega}\\ & \qquad - \left(\textnormal{div} \ \textbf{e}^{k,h}, \textbf{p} - \textbf{i}^k \textbf{I}^h \textbf{p} \right)_{\Omega} + \left( \textnormal{div} (\textbf{u} - \textbf{i}^k \textbf{R}^h \textbf{u}), \boldsymbol{\eta}^{k,h} \right)_{\Omega} \\ & \qquad  +  \left\langle  \sigma_1(\textbf{u}_1 - \textbf{I}^k_1\textbf{R}^h_1\textbf{u}_1, p_1 - I^k_1I^h_1p_1) \cdot \textbf{n}_1 , \textbf{e}_2^{k,h} - \textbf{e}_1^{k,h}\right\rangle_{\Gamma} 
      \\ & \qquad  - \left\langle  \sigma_2 (\textbf{u}_2 - \textbf{I}^k_2\textbf{R}^h_2\textbf{u}_2, p_2 - I^k_2 I^h_2 p_2) \cdot \textbf{n}_2, \textbf{e}_2^{k,h} - \textbf{e}_1^{k,h}\right\rangle_{\Gamma} \\
      & \qquad - \left\langle (\textbf{u}_2 - \textbf{I}^k_2\textbf{R}^h_2\textbf{u}_2) - (\textbf{u}_1 - \textbf{I}^k_1\textbf{R}^h_1\textbf{u}_1),   \sigma_1(\textbf{e}_1^{k,h}, \eta_1^{k,h}) \cdot \textbf{n}_1 \right\rangle_{\Gamma} \hspace*{1.5 cm}
      \\ & \qquad+ \left\langle (\textbf{u}_2 - \textbf{I}^k_2\textbf{R}^h_2\textbf{u}_2) - (\textbf{u}_1 - \textbf{I}^k_1\textbf{R}^h_1\textbf{u}_1),  \sigma_2(\textbf{e}_2^{k,h}, \eta^{k,h}_2) \cdot \textbf{n}_2\right\rangle_{\Gamma} 
      \\ &\qquad + \frac{\gamma}{h}\left\langle (\textbf{u}_2 - \textbf{I}^k_2\textbf{R}^h_2\textbf{u}_2) - (\textbf{u}_1 - \textbf{I}^k_1\textbf{R}^h_1\textbf{u}_1), \textbf{e}_2^{k,h} - \textbf{e}_1^{k,h}\right\rangle_{\Gamma}  \Big\}\diff t.
    \end{aligned}
  \end{equation*}
  Most of these terms were already estimated in previous proofs. We dealt with the time contributions of the time derivative~(\ref{time_equivalence}), the Laplacian terms~(\ref{laplace_time}), and normal derivatives~(\ref{interface_time}) in Theorem~\ref{theorem_semidiscrete_heat}. We looked at the space components in~(\ref{time_derivative_space}) and~(\ref{laplacian_space}) in Theorem~\ref{theorem_discrete_heat}. We considered interface terms in equations~(\ref{heat_interface_estimation}) and~(\ref{heat_interface_estimation_zero}). In~(\ref{heat_interface_estimation}) we have to additionally account for the interpolation in space of the pressure
  \begin{equation*}
    \begin{aligned}
      \int_{I^{n,m}_1} h \big| \big| ( i_1^k p_1 - i_1^k I_1^h p_1) \cdot \textbf{n}_1\big|\big|_{\Gamma}^2 \diff t &\leq  k^{n,m}_1h^{2r} \left|\left|\nabla^r p_1(t^{n,m}_1) \right|\right|_{\Gamma}^2.
    \end{aligned}
  \end{equation*}
  The remaining divergence terms are equal to zero
  \begin{equation*}
    \begin{aligned}
      \int_{I^{n,m}_1}\left(\textnormal{div} \ \textbf{e}^{k,h}_1, p_1 - i_1^k I_1^h p_1\right)_{\Omega_1}\diff t &= 0, \\
      \int_{I^{n,m}_1}\left(\textnormal{div} (\textbf{u}_1 - \textbf{i}_1^k \textbf{R}_1^h \textbf{u}_1), \eta^{k,h}_1 \right)_{\Omega_1}\diff t &= 0.
    \end{aligned}
  \end{equation*}
  Indeed, the exact solution~$\textbf{u}$ and the fully discrete solution~$\textbf{u}^{k,h}$ are divergence-free by definition. The Ritz projection~$\textbf{R}^h \textbf{u}$ is also divergence-free and the time projection operator~$i^k$ does not violate this property, see equation~(\ref{interpolation_div_zero}). That ends the proof.
\end{proof}
Equivalently, we obtained 
\begin{multline*}
  \big| \big| \textbf{e}^{k,h}(t^N)\big|\big| + \int_I \big|\big| \big| \textbf{e}^k\big|\big|\big| \diff t
  \le
  \sum_{j=1}^2 \mathcal{O}\big(k_j\left| \left| \textbf{d}_t \dot{\epsilon} (\textbf{u}_j) \right| \right|_{\Omega_j} \big) + \mathcal{O}\big(k_j \left| \left| d_t p_j \right|\right|_{\Omega_j} \big) \\
  + \mathcal{O}\big(k_j h^{\frac{1}{2}} \left|\left| \textbf{d}_t \sigma_j(\textbf{u}_j, p_j) \cdot \textbf{n}_j  \right|\right|_{\Gamma} \big) \\
  + \mathcal{O} \big(h^{r+1} \left| \left| \textbf{d}_t \nabla^{r + 1} \textbf{u}_j\right| \right|_{\Omega_j} \big)  
  + \mathcal{O}\big(h^r \left| \left|\nabla^{r + 1} \textbf{u}_j \right| \right|_{\Omega_j}\big) +
    \mathcal{O}\big(h^r \left| \left|\nabla^{r} p_j \right| \right|_{\Omega_j}\big).
\end{multline*}
It is another example of an optimal estimate. We again were able to fully decouple time-step dependence.
We will also show a suboptimal estimate of the pressure error.
\begin{theorem}[A priori pressure estimate for the fully discrete Stokes problem]\label{thm:stokes:p}
  Let $\textbf{u} \in X(H^1_0(\Omega))^d$, $\textbf{u}_j \in W^{1, \infty}(H^{r+1}(\Omega_j))^d$, $p_j \in \in W^{1, \infty}(H^{r}(\Omega_j))$ for $j=1,2$ be continuous solutions to~(\ref{continuous_stokes}) and $\textbf{u}^{k,h} \times \textbf{p}^{k,h}\in \big(X^{k,h}(r)\big)^d \times X^{k,h}(r-1)$ their discrete counterparts and solutions to~(\ref{fully_discrete_stokes}), then the following estimate holds
  \begin{align*}
    \int_I  \big|\big|\boldsymbol{\eta}^{k,h} &\big|\big|^2_{\Omega} \diff t 
     \leq C\sum_{j=1}^2 \sum_{n = 1}^{N}\sum_{m = 1}^{N^n_j}  \Bigg\{(k^{n, m}_j)^2 \max_{t \in I} \left| \left| \textbf{d}_t \dot{\epsilon} (\textbf{u}_j) \right| \right|^2_{\Omega_j} + (k^{n,m}_j)^2 \max_{t \in I} \left| \left|d_t p_j\right| \right|_{\Omega_j}^2  \\
    & + (k^{n, m}_j)^2 h \max_{t \in I} \left|\left| \textbf{d}_t \partial _{\textbf{n}_j}\textbf{u}_j \right|\right|_{\Gamma}^2 + (k^{n, m}_j)^2 h \max_{t \in I} \left|\left| d_t p_j\right|\right|_{\Gamma}^2 \\
    &  + h^{2r + 2}  \max_{t \in I} \Big|\Big| \textbf{d}_t \nabla^{r + 1} \textbf{u}_j \Big| \Big|^2_{\Omega_j} + h^{2r} \left| \left|\nabla^{r + 1} \textbf{u}_j(t^{n,m}_j) \right| \right|^2_{\Omega_j} 
    + h^{2r} \left| \left|\nabla^{r} p_j(t^{n,m}_j) \right| \right|^2_{\Omega_j}
    \Bigg\} 
  \end{align*}
  where the errors $\textbf{e}^k = (\textbf{e}_1^k, \textbf{e}_2^k)^T$, $\boldsymbol{\eta}^k = (\eta_1^k, \eta_2^k)^T$ are defined as
  $\textbf{e}_j^{k,h} \coloneqq \textbf{u}_j^{k,h} - \textbf{i}_j^k \textbf{R}_j^h \textbf{u}_j$ and $\eta_j^{k,h} \coloneqq p_j^{k,h} - i_j^k I_j^h p_j$ for $j=1,2$.
  \label{theorem_discrete_stokes_pressure}
\end{theorem}
\begin{proof}
  We would like to obtain an estimate of the form 
  \begin{equation}
    c \int_{I_1^{n,m}}k_1^{n,m}\left| \left| \boldsymbol{\eta}^{k,h}_1\right| \right|_{\Omega}^2 \diff t\leq \int_{I_1^{n,m}} (\textbf{d}_t^k \textbf{e}^{k,h}_1, \textbf{e}^{k,h}_1) \diff t +\int_{I_1^{n,m}} \left|\left| \left|  \textbf{e}^{k,h}_1\right|\right| \right|_{\Omega}^2 \diff t
    \label{pressure_estimate}
  \end{equation}
  on each $I_1^{n,m}$ for $\textbf{e}_1^{k,h}$ as well as an analogous set of estimates for $\textbf{e}_2^{k,h}$. We will be then able to use Theorem~\ref{theorem_discrete_stokes} on the right side of this identity. To achieve this goal, we need to use the inf-sup stability of our trial space. We are going to use a generalized version of the inf-sup stability condition proved in \cite{DG}, from which follows that there exists a constant $\beta$ such that for every $\textbf{q}^{k,h} \in X^{k,h}(r - 1) $, we have 
  \begin{equation*}
    \begin{aligned}
      & \int_{I} ||\textbf{q}^{k,h}||_{\Omega} \diff t \leq \beta \sup_{\boldsymbol{\varphi}^{k,h} \in V^{k,h}}\frac{b(\boldsymbol{\varphi}^{k,h}, \textbf{q}^{k,h})}{|||\boldsymbol{\varphi}^{k,h}|||_{\Omega}} , 
    \end{aligned}
  \end{equation*}
  where $V^{k,h} \coloneqq \left(X^{k,h}(r) \right)^d$. In particular, we can claim that
  \begin{equation*}
    \begin{aligned}
      & \int_{I} ||\boldsymbol{\eta}^{k,h}||_{\Omega} \diff t \leq \beta \sup_{\boldsymbol{\varphi}^{k,h} \in V^{k,h}} \frac{b(\boldsymbol{\varphi}^{k,h}, \boldsymbol{\eta}^{k,h})}{|||\boldsymbol{\varphi}^{k,h}|||_{\Omega}}.
    \end{aligned}
  \end{equation*}
  We can again use the Galerkin orthogonality
  \begin{equation*}
    \begin{aligned}
      b(\boldsymbol{\varphi}^{k,h}, \boldsymbol{\eta}^{k,h}) =&  - a^k(\textbf{e}^{k,h}, \boldsymbol{\varphi}^{k,h}) 
      + a^k(\textbf{u} - \textbf{i}^k \textbf{R}^h \textbf{u}, \boldsymbol{\varphi}^{k,h}) + b(\boldsymbol{\varphi}^{k,h}, \textbf{p} - \textbf{i}^k \textbf{I}^h \textbf{p}).
    \end{aligned}
  \end{equation*}
  We start with an estimation of the first term
  \begin{multline}
      \sup_{\boldsymbol{\varphi}^{k,h} \in V^{k,h}} \frac{a^k(\textbf{e}^{k,h}, \boldsymbol{\varphi}^{k,h})}{|||\boldsymbol{\varphi}^{k,h}|||_{\Omega}} \leq  c \int_{I} \Big\{ \left| \left| \textbf{d}_t^k \textbf{e}^{k,h} \right|\right|_{\Omega}+ \boldsymbol{\nu} \left| \left| \nabla \textbf{e}^{k,h} \right|\right|_{\Omega} 
       + \nu_1h^{\frac{1}{2}}\left| \left| \partial_{\textbf{n}_1} \textbf{e}^{k,h}_1 \right|\right|_{\Gamma} + \\
       \nu_2h^{\frac{1}{2}}\left| \left| \partial_{\textbf{n}_2} \textbf{e}^{k,h}_2 \right|\right|_{\Gamma} + h^{-\frac{1}{2}}(1 + \gamma) \left| \left| \textbf{e}^{k,h}_2 - \textbf{e}^{k,h}_1 \right|\right|_{\Gamma}  \Big\} \diff t \\
        \leq  c \int_{I} \Big\{ \left| \left| \textbf{d}_t^k \textbf{e}^{k,h} \right|\right|_{\Omega}+\left| \left|\left| \textbf{e}^{k,h} \right|\right|\right|_{\Omega} \Big\} \diff t. 
    \label{inf-sup_error}
  \end{multline}
  The time derivative is equal to 
  \begin{equation*}
    \begin{aligned}
      \int_{I_1^{n,m}} \left| \left| \textbf{d}_t^k \textbf{e}^{k,h}_1 \right|\right|_{\Omega_1} \diff t = \left| \left| \textbf{e}^{k,h}_1(t_1^{n,m}) - \textbf{e}^{k,h}_1(t^{n,m - 1}_1)\right|\right|_{\Omega_1}.
    \end{aligned}
  \end{equation*}
  That leads to an estimation 
  \begin{equation*}
    \begin{aligned}
      &  \int_{I_1^{n,m}} ||\boldsymbol{\eta}^{k,h}||_{\Omega_1} \diff t \leq \left| \left| \textbf{e}^{k,h}_1(t_1^{n,m}) - \textbf{e}^{k,h}_1(t^{n,m - 1}_1)\right|\right|_{\Omega_1} + \int_{I_1^{n,m}}\left| \left| \left| \textbf{e}^{k,h}_1 \right|\right|\right|_{\Omega_1} \diff t.
    \end{aligned}
  \end{equation*}
  Knowing that all of these functions are piecewise constant in time, the inequality is equivalent to
  \begin{equation*}
        k_1^{n,m} ||\boldsymbol{\eta}^{k,h}(t_1^{n,m})||_{\Omega_1} \leq \left| \left| \textbf{e}^{k,h}_1(t_1^{n,m}) - \textbf{e}^{k,h}_1(t^{n,m - 1}_1)\right|\right|_{\Omega_1} + k_1^{n,m}\left| \left| \left| \textbf{e}^{k,h}_1(t_1^{n,m}) \right|\right|\right|_{\Omega_1} .
  \end{equation*}
  By squaring both sides of the inequality and going back to the integral form, we get
  \begin{equation*}
     \int_{I_1^{n,m}} k_1^{n,m} ||\boldsymbol{\eta}^{k,h}||_{\Omega_1}^2 \diff t \leq 2\left(\left| \left| \textbf{e}^{k,h}_1(t_1^{n,m}) - \textbf{e}^{k,h}_1(t^{n,m - 1}_1)\right|\right|_{\Omega_1}^2 + \int_{I_1^{n,m}} k_1^{n,m} \left| \left| \left| \textbf{e}^{k,h}_1 \right|\right|\right|_{\Omega_1}^2 \diff t\right).
  \end{equation*}
  Then, based on~(\ref{time_identity}), we have 
  \begin{flalign*}
    \left| \left| \textbf{e}^{k,h}_1(t_1^{n,m}) - \textbf{e}^{k,h}_1(t^{n,m-1}_1)\right|\right|_{\Omega}^2 \leq \int_{I_{1}^{n,m}} 2 (\textbf{d}_t^k \textbf{e}^{k,h}, \textbf{e}^{k,h})_{\Omega_1}.
  \end{flalign*}
  Assuming that $k_1^{n,m} \leq 1$, we have
  \begin{equation*}
    \int_{I_1^{n,m}} k_1^{n,m} \left| \left| \left| \textbf{e}^{k,h}_1 \right|\right|\right|_{\Omega_1}^2 \diff t \leq \int_{I_1^{n,m}} \left| \left| \left| \textbf{e}^{k,h}_1 \right|\right|\right|_{\Omega_1}^2 \diff t.
  \end{equation*}
  This way, we acquire the estimate (\ref{pressure_estimate}).
  We continue with the remaining terms
  \begin{multline*}
     \sup_{\boldsymbol{\varphi}^{k,h} \in V^{k,h}} \frac{ a^k(\textbf{u} - \textbf{i}^k \textbf{R}^h \textbf{u}, \boldsymbol{\varphi}^{k,h})}{|||\boldsymbol{\varphi}^{k,h}|||_{\Omega}} \leq c\int_I \Big\{ \left| \left| \textbf{d}_t^k( \textbf{u} - \textbf{i}^k \textbf{R}^h \textbf{u}) \right|\right|_{\Omega} 
      + \boldsymbol{\nu}\left| \left| \dot{\boldsymbol{\epsilon}}( \textbf{u} - \textbf{i}^k \textbf{R}^h \textbf{u}) \right|\right|_{\Omega} \\ 
       \qquad + h^{\frac{1}{2}}\Big|\Big|\nu_1 \partial_{\textbf{n}_1} (\textbf{u}_1 - \textbf{i}^k_1\textbf{R}^h_1\textbf{u}_1) - \nu_2 \partial_{\textbf{n}_2} (\textbf{u}_2 - \textbf{i}^k_2\textbf{R}^h_2\textbf{u}_2)\Big|\Big|_{\Gamma} \Big\} \diff t.
   \end{multline*}
  All of these terms we estimated in the previous proofs. For a detailed recollection, we refer to the previous proof. The last term can be estimated using
  \begin{equation*}
    \begin{aligned}
      &\sup_{\boldsymbol{\varphi}^{k,h} \in V^{k,h}} \frac{ b(\boldsymbol{\varphi}^{k,h}, \textbf{p} - \textbf{i}^k \textbf{I}^h \textbf{p})}{|||\boldsymbol{\varphi}^{k,h}|||_{\Omega}} \leq c\int_I \Big\{ \left| \left|  \textbf{p} - \textbf{i}^k \textbf{I}^h \textbf{p} \right|\right|_{\Omega} 
      \\ & \qquad \qquad  + h^{\frac{1}{2}}\big| \big| ( p_1 - i_1^k I_1^h p_1)\textbf{n}_1 \big|\big|_{\Gamma} +  h^{\frac{1}{2}}\big| \big| ( p_2 - i_2^k I_2^h p_2) \textbf{n}_2\big|\big|_{\Gamma}\Big\} \diff t.
    \end{aligned}
  \end{equation*}
  
  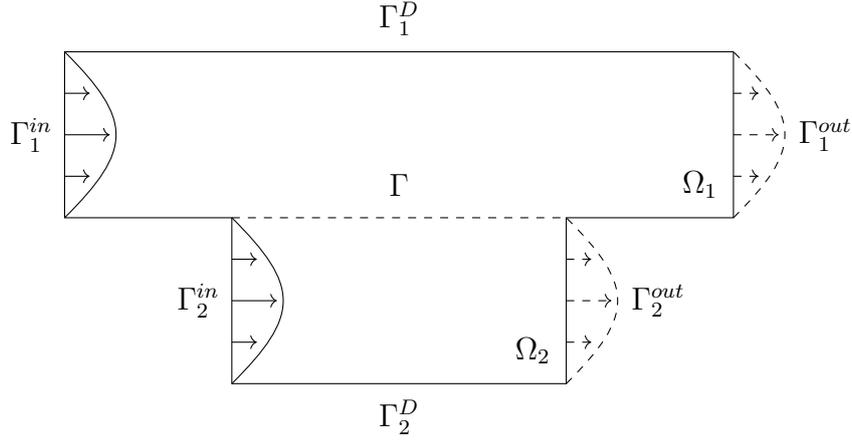
\begin{figure}[t]
    \centering
    \begin{tikzpicture}[scale = 2.2]
      
      \draw (1.0, 0.0) -- (0.0, 0.0) -- (0.0, 1.0) -- (4.0, 1.0) -- (4.0, 0.0) -- (3.0, 0.0) -- (3.0, -1.0) -- (1.0, -1.0) -- (1.0, 0.0);
      \draw[dashed] (1.0, 0.0) -- (3.0, 0.0);
      \draw (0,0) .. controls (0.41, 0.41)and (0.41, 0.59) .. (0,1);
      \draw[dashed] (4,0) .. controls (4.41, 0.41)and (4.41, 0.59) .. (4,1);
      \draw (1,-1) .. controls (1.41, -0.59) and (1.41, -0.41) .. (1,0);
      \draw[dashed] (3,-1) .. controls (3.41, -0.59) and (3.41, -0.41) .. (3,0);
      \draw[->] (0.0, 0.5) -- (0.27, 0.5);
      \draw[->] (0.0, 0.25) -- (0.15, 0.25);
      \draw[->] (0.0, 0.75) -- (0.15, 0.75);

      \draw[->] (1.0, -0.5) -- (1.27, -0.5);
      \draw[->] (1.0, -0.75) -- (1.15, -0.75);
      \draw[->] (1.0, -0.25) -- (1.15, -0.25);

      \draw[dashed, ->] (3.0, -0.5) -- (3.27, -0.5);
      \draw[dashed, ->] (3.0, -0.75) -- (3.15, -0.75);
      \draw[dashed, ->] (3.0, -0.25) -- (3.15, -0.25);

      \draw[dashed, ->] (4.0, 0.5) -- (4.27, 0.5);
      \draw[dashed, ->] (4.0, 0.25) -- (4.15, 0.25);
      \draw[dashed, ->] (4.0, 0.75) -- (4.15, 0.75);
      
      \node at (3.8, 0.2) {\large{$\Omega_1$}};
      \node at (2.8, -0.8) {\large{$\Omega_2$}};
      \node at (-0.2, 0.5) {\large{$\Gamma_1^{in}$}};
      \node at (0.8, -0.5) {\large{$\Gamma_2^{in}$}};
      \node at (2.0, 1.2) {\large{$\Gamma_1^D$}};
      \node at (2.0, -1.2) {\large{$\Gamma_2^D$}};
      \node at (4.55, 0.5) {\large{$\Gamma_1^{out}$}};
      \node at (3.55, -0.5) {\large{$\Gamma_2^{out}$}};
      \node at (2.0, 0.2) {\large{$\Gamma$}};
      
    \end{tikzpicture}
    \caption{We show a sketch of the domains for the Stokes example. The interface is denoted by $\Gamma$. We prescribe parabolic inflows on the inlets $\Gamma_j^{in}$ and free Neumann conditions on the outlets $\Gamma_j^{out}$ for $j=1,2$. Otherwise, we take no-slip boundary conditions on $\Gamma_j^D$.}
    \label{stokes_domain}
  \end{figure}

  For the interface terms, we have 
  \begin{equation*}
      \int_{I^{n,m}_1} \big| \big| ( p_1 - i_1^k I_1^h p_1) \textbf{n}_1 \big|\big|_{\Gamma}^2 \diff t \leq  (k^{n,m}_1)^3 \max_{t \in I} \left|\left| d_t p_1 \right|\right|_{\Gamma}^2  + k^{n,m}_1h^{2r-1} \left|\left|\nabla^r p_1(t^{n,m}_1) \right|\right|_{\Gamma}^2.
  \end{equation*}
  Combining together all of the steps ends the proof.
\end{proof}
We showed a suboptimal estimate of the form
\begin{multline*}
    \int_I  \big|\big|\boldsymbol{\eta}^{k,h} \big|\big|_{\Omega} \diff t 
    = \sum_{j=1}^2 \mathcal{O}\big(k_j^{\frac{1}{2}}\left| \left| \textbf{d}_t \dot{\epsilon} (\textbf{u}_j) \right| \right|_{\Omega_j}\big) + \mathcal{O}\big(k_j^{\frac{1}{2}} \left| \left| d_t p_j \right|\right|_{\Omega_j} \big)   
    + \mathcal{O}\big(k_j^{\frac{1}{2}}h^{\frac{1}{2}}\left|\left|\textbf{d}_t \partial_{\textbf{n}_j}\textbf{u}_j\right|\right|_{\Gamma} \big) \\
    +\mathcal{O}\big(k_j^{\frac{1}{2}}h^{\frac{1}{2}} \left|\left|d_t p_j\right|\right|_{\Gamma}\big)\\
     +  \mathcal{O}\big( k_j^{-\frac{1}{2}}h^{r + 1}  \big|\big| \textbf{d}_t \nabla^{r + 1} \textbf{u}_j \big| \big|_{\Omega_j}\big) + \mathcal{O}\big(k_j^{-\frac{1}{2}}h^r \left| \left|\nabla^{r + 1} \textbf{u}_j \right| \right|_{\Omega_j} \big)  
    + \mathcal{O}\big(k_j^{-\frac{1}{2}}h^r \left| \left|\nabla^{r} p_j \right| \right|_{\Omega_j}\big)
\end{multline*}
Due to the inf-sup estimation~(\ref{inf-sup_error}), we lost half an order of convergence in time. Namely, the source of this loss is the time derivative. We obtained the term $||\textbf{d}_t^k \textbf{e}^{k,h}||^2_{\Omega} = (\textbf{d}_t^k \textbf{e}^{k,h}, \textbf{d}_t^k \textbf{e}^{k,h})_{\Omega}$, whereas, on the left side of the orthogonality identity~(\ref{fully-discrete-stokes-orthogonal}), we have $(\textbf{d}_t^k \textbf{e}^{k,h}, \textbf{e}^{k,h})_{\Omega}$. We were only able to show the estimate 
\begin{equation*}
  k (\textbf{d}_t^k \textbf{e}^{k,h}, \textbf{d}_t^k \textbf{e}^{k,h})_{\Omega} \leq (\textbf{d}_t^k \textbf{e}^{k,h}, \textbf{e}^{k,h})_{\Omega}.
\end{equation*}
This is a nontrivial problem and was encountered for example in the series of articles~\cite{NS_analysis_1, NS_analysis_2, NS_analysis_3, NS_analysis_4} about the Navier-Stokes equations. Specifically, in~\cite{NS_analysis_4} the authors comment on the difficulties that come with the optimal estimation of the time derivative term. This issue has been successfully circumvented in~\cite{pressure_optimal}, where the optimal convergence rate of pressure was retrieved. Indeed, the authors were able to show optimality for the Crank-Nicolson time-stepping scheme in $L^2$ 
\begin{equation}
  \left| \left| I^k p - p^k \right| \right|_{L^2\left(I, H^1(\Omega) \right)} \leq C k^2
  \label{pressure_L2}
\end{equation}
and $L^\infty$ 
\begin{equation}
  \left| \left| J^k p - p^k \right| \right|_{L^{\infty}\left(I, H^1(\Omega) \right)} \leq C k^2
  \label{pressure_infinity}
\end{equation}
norms. The operator $J^k$ is given by 
\begin{equation*}
  J^k p \big|_{I^n} \coloneqq u(\bar{t}_{n}),
\end{equation*}
where $\bar{t}_{n} \coloneqq \frac{t_{n} + t_{n - 1}}{2}$. This publication considered neither coupled problems nor multirate time-stepping.



\subsection{Numerical Example}

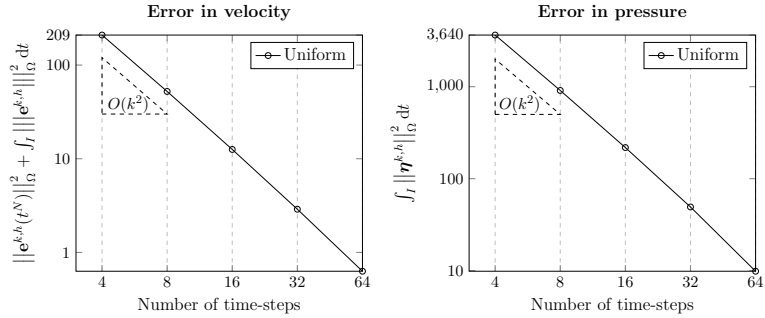
\begin{figure}[t]
  \begin{center}
    \begin{tikzpicture}[scale = 0.55]
      \begin{loglogaxis}[
	  title style={align=center},
	  title={\large{\textcolor{white}{adjust to the next line}} \\ \large{\textbf{Error in velocity}}},
	  ylabel={\large{$\big| \big| \textbf{e}^{k,h}(t^N)\big|\big|^2_{\Omega}   + \int_I \big|\big|\big|\textbf{e}^{k,h} \big|\big|\big|^2_{\Omega} \diff t$}},
	  xlabel={\large{Number of time-steps}},
          xmax=64,
	  ymin=0.6293964200027503, 
	  ymax=209.12388873964287,
          log ticks with fixed point,
          xtick={4, 8, 16, 32, 64},
          ytick={1, 10, 100, 209},
	  legend pos=north east,
	  xmajorgrids=true,
	  grid style=dashed,
	]

	\addplot[
	  color=black,
	  mark=o,
	]
	coordinates {
	  (4.0, 209.12388873964287)
          (8.0, 52.276864680267536)
          (16.0, 12.56609938782985)
          (32.0, 2.8908205774794182)
          (64.0, 0.6293964200027503)
          
	};
        
        \draw[dashed] (axis cs:4,30) -- (axis cs:8,30) -- (axis cs:4,120) -- (axis cs:4,30);
	\node at (axis cs:5.3,38) {$O(k^2)$};
	\legend{\large{Uniform}}
	
      \end{loglogaxis}
    \end{tikzpicture}
    \begin{tikzpicture}[scale = 0.55]
      \begin{loglogaxis}[
	  title style={align=center},
	  title={\large{\textcolor{white}{adjust to the next line}} \\ \large{\textbf{Error in pressure}}},
	  ylabel={\large{$\int_I \big| \big| \boldsymbol{\eta}^{k,h}\big|\big|^2_{\Omega}\diff t$}},
	  xlabel={\large{Number of time-steps}},
          xmax=64,
	  ymin=9.858885309738726, 
	  ymax=3644.0548987697075,
          log ticks with fixed point,
          xtick={4, 8, 16, 32, 64},
          ytick={10, 100, 1000, 3644},
	  legend pos=north east,
	  xmajorgrids=true,
	  grid style=dashed,
	]

	\addplot[
	  color=black,
	  mark=o,
	]
	coordinates {
	  (4.0, 3644.0548987697075)
          (8.0, 910.1062745605058)
          (16.0, 217.92295497690716)
          (32.0, 49.2770263304117)
          (64.0, 9.858885309738726)

	};

        \draw[dashed] (axis cs:4,500) -- (axis cs:8,500) -- (axis cs:4,2000) -- (axis cs:4,500);
	\node at (axis cs:5.2,650) {$O(k^2)$};	
	\legend{\large{Uniform}}
	
      \end{loglogaxis}
    \end{tikzpicture}
  \end{center}
  \caption{Coupled Stokes problem on uniform time meshes. We show the squared error in velocity (left) and pressure (right). For velocity and pressure we obtain second order of convergence (in the squared error). While this is optimal for the velocity, it shows that the pressure estimate is suboptimal.}
  \label{fig:global}
\end{figure}

We consider a coupled Stokes problem on a domain consisting of two pipelines  $\Omega_1=(0,4)\times (0,1)$ and $\Omega_2=(1,3)\times (-1,0)$ connected by the interface $\Gamma = (1,3)\times \{0\}$. The viscosities are taken as $\nu_1 = 1$ and $\nu_2 = 56$. This ratio is similar to the viscosity ratio of water an oil and hence, we will call $\Omega_1$ the ``water problem'' and $\Omega_2$ the ``oil problem''.
On each of the inlets $\Gamma^{in}_j$ we prescribe parabolic inflows
$\textbf{u}_1^{in}(x,y) = \sin(\pi t) y (1 - y)$ and $\textbf{u}_2^{in}(x,y) = \sin(\pi t) y (1 + y)$ for $t \in I=[0,1]$. On the outlets $\Omega^{out}_j$, we choose free Neumann boundary conditions. Otherwise, we take no-slip Dirichlet boundary conditions. We show a sketch of the domain in Figure~\ref{stokes_domain}. Since the flow is fully driven by the boundary conditions, we take $\textbf{f}_1 = \textbf{f}_2 = \textbf{0}$.

In Fig.~\ref{fig:global} we show the velocity and pressure error on uniformly refined time meshes with the same number of time-steps in both subproblems. The spatial mesh is kept fixed at high resolution. Both velocity and pressure converge with linear order, which is optimal for the velocity, see Theorem~\ref{thm:stokes:vel}. For the pressure Theorem~\ref{thm:stokes:p} only showed a suboptimal convergence and we refer to the discussion at the end of the previous section. 


Next we analyze the exactness of the multirate error estimates and start with very coarse time meshes with only 4 time-steps in both of the domains. Then we refine only one of the two domains $\Omega_1$ or $\Omega_2$. Fig.~\ref{velocity_convergence_rate} shows the convergence rate in the velocity error. In the upper row we only refine the first domain $\Omega_1$ while the lower row shows time-refinement in $\Omega_2$ only. The left column indicates the total velocity error spanning over both domains while the right column gives the error only for that domain, where time-mesh refinement takes place. The triangles again indicate quadratic convergence (of the squared errors). We first observe that the variational multirate method is well localizing the errors and that refinement in only one domain is indeed sufficient to reduce the error in that domain only. As expected the overall error is dominated by the ``oil problem'' and the total error will only decrease, if the time mesh corresponding to $\Omega_2$ is refined. The results in Fig.~\ref{velocity_convergence_rate} further show that the error estimates are well able to localize the error to the two domains which validates our findings in Theorem~\ref{theorem_discrete_stokes} where we were able to fully decouple time contributions from different subproblems.

\begin{figure}[t]

  \begin{center}
    \begin{tabular}{cc}
      \begin{tikzpicture}[scale = 0.55]
        \begin{semilogxaxis}[
	    title style={align=center},
	    title={\large{\textbf{Multirate time-stepping}}\\\large{\textbf{in the water problem}}},
	    ylabel={\large{$\big| \big| \textbf{e}^{k,h}(t^N)\big|\big|^2_{\Omega}   + \int_I \big|\big|\big|\textbf{e}^{k,h} \big|\big|\big|^2_{\Omega} \diff t$}},
	    xlabel={\large{Number of time-steps in the water problem}},
            xmax=64,
	    ymin=204, 
	    ymax=214,
            log ticks with fixed point,
            xtick={4, 8, 16, 32, 64},
            ytick={1, 10, 100, 209, 300},
	    legend pos=north east,
	    xmajorgrids=true,
	    grid style=dashed,
	  ]
	  
	  \addplot[
	    color=black,
	    mark=o,
	  ]
	  coordinates {
            (4.0, 209.12388873964287)
            (8.0, 209.0224192380493)
            (16.0, 208.99610399013505)
            (32.0, 208.9893469162537)
            (64.0, 208.98762123917425)
	  };

          \draw[dashed] (axis cs:4,210.8) -- (axis cs:8,210.8) -- (axis cs:4,213.25) -- (axis cs:4,210.8);
	  
	  \legend{\large{Asymmetric}}
	  
        \end{semilogxaxis}
      \end{tikzpicture}
      &
      \begin{tikzpicture}[scale = 0.55]
        \begin{loglogaxis}[
	    title style={align=center},
	    title={\large{\textbf{Multirate time-stepping}}\\\large{\textbf{in the water problem}}},
	    ylabel={\large{$\int_I \nu_1^2 \big|\big| \nabla \textbf{e}^{k,h}_1 
                \big|\big|^2_{\Omega_1} \diff t$}},
	    xlabel={\large{Number of time-steps in the water problem}},
            xmax=64,
	    ymin=0.00035664009997386385, 
	    ymax=0.13160419459484587,
            log ticks with fixed point,
            xtick={4, 8, 16, 32, 64},
            ytick={0.001, 0.01, 0.13},
	    legend pos=north east,
	    xmajorgrids=true,
	    grid style=dashed,
	  ]

	  \addplot[
	    color=black,
	    mark=o,
	  ]
	  coordinates {
            (4.0, 0.13160419459484587)
            (8.0, 0.032847327230317885)
            (16.0, 0.007856348069003946)
            (32.0, 0.0017729822387376145)
            (64.0, 0.0003537318395834676)
	  };

          \draw[dashed] (axis cs:4, 0.0185) -- (axis cs:8,0.0185) -- (axis cs:4,0.074) -- (axis cs:4, 0.0185);
	  
	  \legend{\large{Asymmetric}}
	  
        \end{loglogaxis}
      \end{tikzpicture}
      \\
      \begin{tikzpicture}[scale = 0.55]
        \begin{loglogaxis}[
	    title style={align=center},
	    title={\large{\textbf{Multirate time-stepping}}\\\large{\textbf{in the oil problem}}},
	    ylabel={\large{$\big| \big| \textbf{e}^{k,h}(t^N)\big|\big|^2_{\Omega}   + \int_I \big|\big|\big|\textbf{e}^{k,h} \big|\big|\big|^2_{\Omega} \diff t$}},
	    xlabel={\large{Number of time-steps in the oil problem}},
            xmax=64,
	    ymin=0.7517021764173385, 
	    ymax=209.12388873964287,
            log ticks with fixed point,
            xtick={4, 8, 16, 32, 64},
            ytick={1, 10, 100, 209},
	    legend pos=north east,
	    xmajorgrids=true,
	    grid style=dashed,
	  ]
	  
	  \addplot[
	    color=black,
	    mark=o,
	  ]
	  coordinates {
            (4.0, 209.12388873964287)
            (8.0, 52.37001928638283)
            (16.0, 12.682143249731393)
            (32.0, 3.0120885756425473)
            (64.0, 0.7517021764173385)
	  };

          \draw[dashed] (axis cs:4,30) -- (axis cs:8,30) -- (axis cs:4,120) -- (axis cs:4,30);
	  
	  \legend{\large{Asymmetric}}
	  
        \end{loglogaxis}
      \end{tikzpicture}
      &
      \begin{tikzpicture}[scale = 0.55]
        \begin{loglogaxis}[
	    title style={align=center},
	    title={\large{\textbf{Multirate time-stepping}}\\\large{\textbf{in the oil problem}}},
	    ylabel={\large{$\int_I \nu_2^2\big|\big| \nabla \textbf{e}^{k,h}_2 
                \big|\big|^2_{\Omega_2} \diff t$}},
	    xlabel={\large{Number of time-steps in the oil problem}},
            xmax=64,
	    ymin=0.5654988185936245, 
	    ymax=208.93749622522043,
            log ticks with fixed point,
            xtick={4, 8, 16, 32, 64},
            ytick={1, 10, 100, 209},
	    legend pos=north east,
	    xmajorgrids=true,
	    grid style=dashed,
	  ]
	  
	  \addplot[
	    color=black,
	    mark=o,
	  ]
	  coordinates {
            (4.0, 208.93749622522043)
            (8.0, 52.18372470408981)
            (16.0, 12.495900542639403)
            (32.0, 2.825872091184297)
            (64.0, 0.5654988185936245)
	  };

          \draw[dashed] (axis cs:4,30) -- (axis cs:8,30) -- (axis cs:4,120) -- (axis cs:4,30);
	  
	  \legend{\large{Asymmetric}}
	  
        \end{loglogaxis}
      \end{tikzpicture}
    \end{tabular}
  \end{center}
  \caption{Convergence rate of the total velocity error on a uniform mesh (top), total and $H^1$ water errors for refinement in the water mesh only (middle row), and total and $H^1$ oil errors for refinement in the oil mesh only (bottom row) with respect to the number of time-steps.}
  \label{velocity_convergence_rate}
\end{figure}
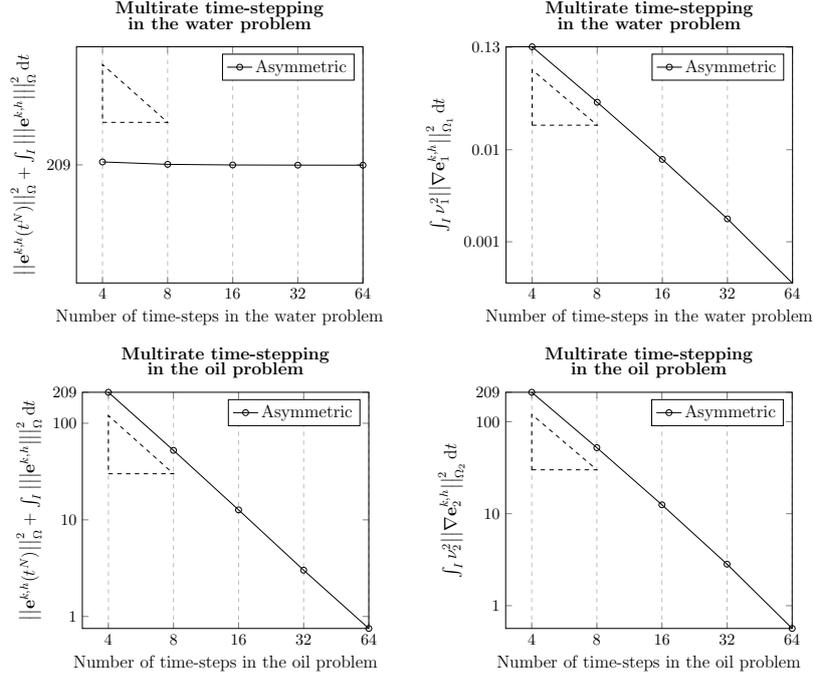

Finally, in Figure~\ref{pressure_convergence_rate} we present convergence graphs corresponding to the pressure. These results show that the two subproblems are not fully decoupled. Indeed, especially in the case of the water problem, we can see a deterioration of the convergence rate in the case of asymmetric time meshes. The graphs suggest that this deterioration is driven by the number of micro time-steps. This effect is much more pronounced in the water problem. That might be due to the difference in viscosity between the two problems.

\begin{figure}[t]

  \begin{center}
    \begin{tabular}{cc}
      \begin{tikzpicture}[scale = 0.55]
        \begin{loglogaxis}[
	    title style={align=center},
	    title={\large{\textbf{Multirate time-stepping}}\\\large{\textbf{in the water problem}}},
	    ylabel={\large{$\int_I \big| \big| \boldsymbol{\eta}^{k,h}\big|\big|^2_{\Omega}\diff t$}},
	    xlabel={\large{Number of time-steps in the water problem}},
            xmax=64,
	    ymin= 49.334723083002245, 
	    ymax=50000,
            log ticks with fixed point,
            xtick={4, 8, 16, 32, 64},
            ytick={10, 100, 1000, 3644},
	    legend pos=north east,
	    xmajorgrids=true,
	    grid style=dashed,
	  ]
	  
	  \addplot[
	    color=black,
	    mark=o,
	  ]
	  coordinates {
 	    (4.0, 3644.0548987697075)
            (8.0, 3648.1688058287123)
            (16.0, 3650.785870274405)
            (32.0, 3652.242192273331)
            (64.0, 3653.007901430062)
	  };

 	  \addplot[
	    color=black,
	    mark=square,
	  ]
	  coordinates {
            (8.0, 3648.1688058287123)
            (16.0, 911.1294387441708)
            (32.0, 218.1707017554819)
            (64.0, 49.334723083002245)
	  };

          \draw[dashed] (axis cs:4,5500) -- (axis cs:8,5500) -- (axis cs:4,29000) -- (axis cs:4,5500);
	  
	  \legend{\large{Asymmetric}, \large{Uniform}}
	  
        \end{loglogaxis}
      \end{tikzpicture}
      &
      \begin{tikzpicture}[scale = 0.55]
        \begin{loglogaxis}[
	    title style={align=center},
	    title={\large{\textbf{Multirate time-stepping}}\\\large{\textbf{in the water problem}}},
	    ylabel={\large{$\int_I \big| \big|       \eta^{k,h}_1\big|\big|^2_{\Omega_1}\diff t$}},
	    xlabel={\large{Number of time-steps in the water problem}},
            xmax=64,
	    ymin=0.010300577097754874, 
	    ymax=3.4855402362208423,
            log ticks with fixed point,
            xtick={4, 8, 16, 32, 64},
            ytick={0.1, 1.0, 3.49},
	    legend pos=north east,
	    xmajorgrids=true,
	    grid style=dashed,
	  ]
          
	  \addplot[
	    color=black,
	    mark=o,
	  ]
	  coordinates {
            (4.0, 3.4855402362208423)
            (8.0, 0.9317804151139482)
            (16.0, 0.2566173357430428)
            (32.0, 0.07705017036833921)
            (64.0, 0.027352254959012605)
	  };
          
	  \addplot[
	    color=black,
	    mark=square,
	  ]
	  coordinates {
            (8.0, 0.9317804151139482)
            (16.0, 0.22502265697789084)
            (32.0, 0.05115428460041787)
            (64.0, 0.010300577097754874)

	  };

          \draw[dashed] (axis cs:4,0.525) -- (axis cs:8,0.525) -- (axis cs:4,2.1) -- (axis cs:4,0.525);
	  
	  \legend{\large{Asymmetric}, \large{Uniform}}
	  
        \end{loglogaxis}
      \end{tikzpicture}
      \\
      \begin{tikzpicture}[scale = 0.55]
        \begin{loglogaxis}
          [
     	    title style={align=center},
            title={\large{\textbf{Multirate time-stepping}}\\\large{\textbf{in the oil problem}}},
    	    ylabel={\large{$\int_I \big| \big| \boldsymbol{\eta}^{k,h}\big|\big|^2_{\Omega}\diff t$}},
    	    xlabel={\large{Number of time-steps in the oil problem}},
            xmax=64,
    	    ymin=9.855477597969582, 
    	    ymax=3644.0548987697075,
            log ticks with fixed point,
                xtick={4, 8, 16, 32, 64},
                ytick={10, 100, 1000, 3644},
    	    legend pos=north east,
    	    xmajorgrids=true,
    	    grid style=dashed,
    	  ]
  	  \addplot[
  	    color=black,
  	    mark=o,
  	  ]
  	  coordinates {
    	    (4.0, 3644.0548987697075)
              (8.0, 909.6881902771019)
              (16.0, 218.96404929998275)
              (32.0, 51.4984253444735)
              (64.0, 12.786633646096856)
  	  };
  
   	  \addplot[
  	    color=black,
  	    mark=square,
  	  ]
  	  coordinates {
              (8.0, 909.6881902771019)
              (16.0, 217.78945094666994)
              (32.0, 49.247943946063835)
              (64.0, 9.855477597969582)
  	  };
  
            \draw[dashed] (axis cs:4,500) -- (axis cs:8,500) -- (axis cs:4,2000) -- (axis cs:4,500);
  	  
  	  \legend{\large{Asymmetric}, \large{Uniform}}
          \end{loglogaxis}
      \end{tikzpicture}
      &
      \begin{tikzpicture}[scale = 0.55]
        \begin{loglogaxis}[
	    title style={align=center},
	    title={\large{\textbf{Multirate time-stepping}}\\\large{\textbf{in the oil problem}}},
	    ylabel={\large{$\int_I \big| \big| \eta^{k,h}_2\big|\big|^2_{\Omega_2}\diff t$}},
	    xlabel={\large{Number of time-steps in the oil problem}},
            xmax=64,
	    ymin=9.301639791454074, 
	    ymax=3640.5693585334866,
            log ticks with fixed point,
            ytick={10, 100, 1000, 3640},
            xtick={4, 8, 16, 32, 64},
	    legend pos=north east,
	    xmajorgrids=true,
	    grid style=dashed,
	  ]
	  
	  \addplot[
	    color=black,
	    mark=o,
	  ]
	  coordinates {
            (4.0, 3640.5693585334866)
            (8.0, 906.2029615220773)
            (16.0, 215.4789605943768)
            (32.0, 48.013400960696586)
            (64.0, 9.301639791454074)
	  };

 	  \addplot[
	    color=black,
	    mark=square,
	  ]
	  coordinates {
            (8.0, 906.2029615220773)
            (16.0, 216.91598003447692)
            (32.0, 49.03823462355506)
            (64.0, 9.807960898924534)
	  };

          \draw[dashed] (axis cs:4,500) -- (axis cs:8,500) -- (axis cs:4,2000) -- (axis cs:4,500);
	  
	  \legend{\large{Asymmetric}, \large{Uniform}}
	  
        \end{loglogaxis}
      \end{tikzpicture}
    \end{tabular}
  \end{center}
  \caption{Convergence rate of the total pressure error on a uniform mesh (top), total and water errors for refinement in the water mesh (middle row), and total and oil errors for refinement in the oil mesh (bottom row) with respect to the number of time-steps.}
  
  \label{pressure_convergence_rate}
\end{figure}
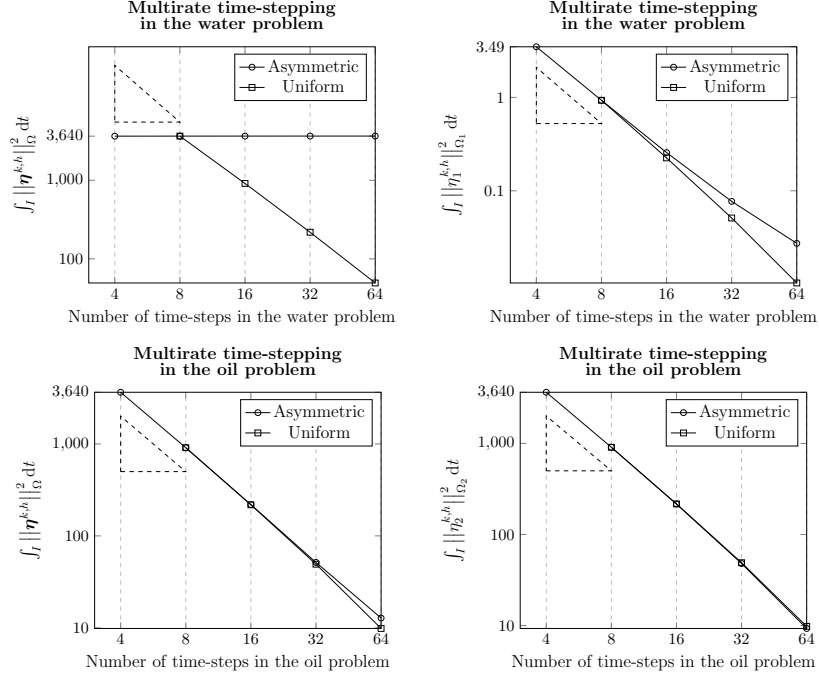

\section{Conclusion}

 In this paper, we proved stability error estimates of the implicit Euler time-stepping scheme for coupled systems with non-matching time interval partitionings. We considered three problems - a simple system of ordinary differential equations as well as two systems of partial differential equations, that is either two heat or two Stokes equations coupled together. We examined both semi-discrete as well as fully discrete formulations. The proofs showed optimal convergence rates for all of them except for the pressure error of the fully discrete Stokes problem. The key components of the proofs were using appropriate projection operators in time as well as choosing coupling conditions that ensured coercivity of the problems. In the case of the fully discrete Stokes equation, we used a  generalized inf-sup condition to account for the coupling conditions. 

A natural extension of our findings would be to consider other time-stepping schemes as well as systems consisting of equations of a different type, for example, a heat equation coupled with a wave equation or a coupling of a Stokes and linear elasticity equations. As a further step, we could also consider nonlinear systems.  

\section{Acknowledgments}
Both authors acknowledge support by the Deutsche Forschungsgemeinschaft (DFG, German Research Foundation) - 314838170, GRK 2297 MathCoRe. TR further acknowledge supported by the Federal Ministry of Education and Research of Germany (project number 05M16NMA).


\end{document}